\documentclass[reqno,final]{amsart}
\usepackage{natbib}  %Nature-like bibliography
\usepackage{fancyhdr} %Headers
\usepackage{color} %Color definition
\usepackage{hyperref} %Internal and external links
\usepackage{graphicx} %Graphics inclusion

\usepackage{amssymb}

\definecolor{aleacolor}{rgb}{0.16,0.59,0.78}

\hypersetup{
breaklinks,
colorlinks=true,
linkcolor=aleacolor,
urlcolor=aleacolor,
citecolor=aleacolor}

\renewcommand{\headheight}{11pt}
\renewcommand{\cite}{\citet}

\theoremstyle{plain}
\newtheorem{theorem}{Theorem}[section]                      
\newtheorem{proposition}[theorem]{Proposition}    
\newtheorem{lemma}[theorem]{Lemma}
\newtheorem{corollary}[theorem]{Corollary}

\theoremstyle{definition}

\theoremstyle{remark}
\newtheorem{remark}[theorem]{Remark}
\newtheorem{example}[theorem]{Example}

\makeatletter \@addtoreset{equation}{section} \makeatother
\renewcommand\theequation{\thesection.\arabic{equation}}
\renewcommand\thefigure{\thesection.\arabic{figure}}
\renewcommand\thetable{\thesection.\arabic{table}}

\DeclareSymbolFont{bbold}{U}{bbold}{m}{n}
\DeclareMathSymbol{\bblambda}{\mathord}{bbold}{21} % mathbb \lambda
\DeclareMathSymbol{\bbone}{\mathord}{bbold}{49}    % mathbb 1

\newcommand{\CC}{\mathbb{C}}
\newcommand{\RR}{\mathbb{R}}
\newcommand{\NN}{\mathbb{N}}
\newcommand{\ZZ}{\mathbb{Z}}
\newcommand{\Zpl}{\ZZ_+}
\newcommand{\overlZpl}{\overline{\ZZ}_+}

\newcommand{\cala}{\mathcal{A}}
\newcommand{\calb}{\mathcal{B}}
\newcommand{\calf}{\mathcal{F}}
\newcommand{\calm}{\mathcal{M}}
\newcommand{\cals}{\mathcal{S}}

\newcommand{\frakX}{\mathfrak{X}}
\newcommand{\fraks}{\mathfrak{s}}
\newcommand{\dd}{{\mathrm{d}}}
\newcommand{\ee}{{\mathrm e}}
\newcommand{\pii}{{\mathrm{\pi}}}

\newcommand{\ab}[1]{\lvert#1\rvert}            % absolute value 
\newcommand{\Ab}[1]{\Bigl\lvert#1\Bigr\rvert}  % absolute value (Big) 

\newcommand{\bfcdot}{{\boldsymbol{\cdot}}}

\newcommand{\card}[1]{\lvert#1\rvert}       % cardinality of a set

\newcommand{\Ch}{\mathrm{Ch}}               % Charlier polynomial

\newcommand{\compl}{\mathrm{c}}     % Complement set

\newcommand{\CPo}{\mathrm{CPo}}     % Compound Poisson distribution

\newcommand{\dirac}{\delta}         % Dirac measure

\newcommand{\dtv}{d_{\mathrm{TV}}}  % total variation distance       

\renewcommand{\geq}{\geqslant}

\renewcommand{\leq}{\leqslant}

\newcommand{\norm}[1]{\lVert#1\rVert}            % norm
\newcommand{\Norm}[1]{\Bigl\lVert#1\Bigr\rVert}  % norm (Big)

\newcommand{\po}{\mathrm{po}}    % Poisson distr. (counting density)
\newcommand{\Po}{\mathrm{Po}}    % Poisson distribution

\newcommand{\set}[1]{\underline{#1}}
\newcommand{\setn}[1]{\underline{#1}_0} 
 
\newcommand{\tvm}[1]{\lvert#1\rvert}   % total variation measure

\newcommand{\unitvec}[1]{e_{#1}^{}}

\newcommand{\vecsum}[1]{\lvert#1\rvert}% sum of components of a vector

\newcommand{\binomial}[2]{\genfrac{(}{)}{0pt}{}{#1}{#2}}

\scrollmode
\begin{document}
%%%%%%%%%%%%%%%%%%%%%%%%%%%%%%%%%%%%%%%%%%%%%%%%%%%%%%%%%%%%%%%%%%%%%%
%%%%%%%%%%%%%%%%%%%%%%%%%%%%%%%%%%%%%%%%%%%%%%%%%%%%%%%%%%%%%%%%%%%%%%
\title[Multivariate and compound Poisson approximation]{Refined 
total variation bounds in the multivariate and compound 
Poisson approximation}

\author{Bero Roos}

\address{FB IV -- Department of Mathematics, \newline
University of Trier, \newline
54286 Trier, Germany.}

\email{bero.roos@uni-trier.de}
\urladdr{http://www.math.uni-trier.de/{\textasciitilde}roos/}

%\thanks{}

\subjclass[2000]{% Better: use the 2010 version 
Primary 60F05;   % Central limit and other weak theorems
secondary 60G50, % Sums of independent random variables; random walks
62E17.           % Approximations to distributions (nonasymptotic)
} 
\keywords{% arranged in alphabetical order
Convolution,
generalized multinomial distribution,
Kerstan's method,
measurable Abelian semigroup,
Poisson process approximation, 
signed measure,
smoothness inequality.
}

\begin{abstract}
We consider the approximation of a convolution of possibly different 
probability measures by (compound) Poisson distributions and also by 
related signed measures of higher order. We present new total 
variation bounds having a better structure than those from the 
literature. A numerical example illustrates the usefulness of the 
bounds, and an application in the Poisson process approximation is 
given. The proofs use arguments from \citet{MR0165555} and
\citet{MR1701409} in combination with new smoothness inequalities, 
which could be of independent interest. 
\end{abstract}

\maketitle
%%%%%%%%%%%%%%%%%%%%%%%%%%%%%%%%%%%%%%%%%%%%%%%%%%%%%%%%%%%%%%%%%%%%%%
\section{Introduction}
%%%%%%%%%%%%%%%%%%%%%%%%%%%%%%%%%%%%%%%%%%%%%%%%%%%%%%%%%%%%%%%%%%%%%%
%%%%%%%%%%%%%%%%%%%%%%%%%%%%%%%%%%%%%%%%%%%%%%%%%%%%%%%%%%%%%%%%%%%%%%
\subsection{Aim of the paper}
%%%%%%%%%%%%%%%%%%%%%%%%%%%%%%%%%%%%%%%%%%%%%%%%%%%%%%%%%%%%%%%%%%%%%%
%%%%%%%%%%%%%%%%%%%%%%%%%%%%%%%%%%%%%%%%%%%%%%%%%%%%%%%%%%%%%%%%%%%%%%
Nowadays, there are numerous results in the (compound) Poisson 
approximation of convolutions of probability distributions
(cf.\ \citealp{MR871856}; \citealp{MR1163825}). 
However, it turned out that the 
investigation in the multidimensional case is somewhat difficult. 
Even in the simple case of Poisson approximation of the generalized 
multinomial distribution, the correct order of approximation is not 
exactly 
known. Indeed, to the best of our knowledge, the literature does not 
contain any lower and upper total variation bounds differing only by 
an absolute constant factor. The problem here is that not only the 
number of convolution factors but also the dimension can be 
arbitrarily large. But there are useful approximation results, 
see, e.g., \cite{MR0160279}, U.\ \cite{MR0179820}, \cite{MR935295}, 
\cite{MR1701409}, \cite{MR2205332}. In this paper, we show how some 
bounds from \cite{MR1701409} can be further substantially improved. 
We also indicate 
how these improved bounds in combination with ideas in 
\cite{MR2310979} can be useful in the compound Poisson approximation. 

The paper is organized as follows. In the next three subsections, we 
explain the notation, comment on the method used,
give a review of some results from the 
literature and discuss the benefits of some result of the 
present paper.
Sections \ref{mainres} and \ref{s1545} are devoted to the main 
results and an application in the Poisson process approximation.  
In Section \ref{Proofs}, we present some auxiliary norm estimates 
including smoothness inequalities as well as the proofs 
of the results. 
%%%%%%%%%%%%%%%%%%%%%%%%%%%%%%%%%%%%%%%%%%%%%%%%%%%%%%%%%%%%%%%%%%%%%%
%%%%%%%%%%%%%%%%%%%%%%%%%%%%%%%%%%%%%%%%%%%%%%%%%%%%%%%%%%%%%%%%%%%%%%
\subsection{Notation}
%%%%%%%%%%%%%%%%%%%%%%%%%%%%%%%%%%%%%%%%%%%%%%%%%%%%%%%%%%%%%%%%%%%%%%
%%%%%%%%%%%%%%%%%%%%%%%%%%%%%%%%%%%%%%%%%%%%%%%%%%%%%%%%%%%%%%%%%%%%%%
In what follows, let $(\frakX,+,\cala)$ be a measurable 
Abelian semigroup with zero element,  
that is, $(\frakX,+)$ is a commutative semigroup with identity 
element~$0$ 
and $\cala$ is a $\sigma$-algebra of subsets of $\frakX$  such that 
the mapping $\frakX\times\frakX\ni(x,y)\mapsto x+y\in\frakX$ from 
$(\frakX\times\frakX,\cala\otimes\cala)$ to $(\frakX,\cala)$ is 
measurable. 
In particular this implies that, for arbitrary $y\in\frakX$, 
the mapping $\frakX\ni x\mapsto x+y\in\frakX$ is measurable as well. 
The approach used in this paper requires 
a measure theoretic setting. Random variables are rarely needed or 
used.
Let $\calf$ (resp.~$\calm$) be the set of all probability 
distributions (resp.~finite signed measures) on $(\frakX,\cala)$. 
Products and powers of finite signed measures in $\calm$ are defined 
in the convolution sense, that is, for $V,W\in\calm$ and $A\in\cala$, 
we write 
\begin{align*}
VW(A)=\int_\frakX V(\{y\in\frakX\,|\,x+y\in A\})\,\dd W(x).
\end{align*}
Empty products and powers of signed measures in $\calm$ are defined 
to be $\dirac_0$, where  $\dirac_x$  is the Dirac measure at 
point $x\in\frakX$. Let $V=V^{+}-V^{-}$
denote the Hahn-Jordan decomposition of~$V\in\calm$ and let
$\tvm{V}=V^{+}+V^{-}$ be its total variation measure. The total
variation norm of $V$ is defined by  $\norm{V}=\tvm{V}(\frakX)$. 
We note that the total variation distance between two finite signed 
measures $V,W\in\calm$ is usually defined by 
$\dtv(V,W)=\sup_{A\in\cala}\ab{V(A)-W(A)}$. 
However, this distance is rarely needed or used here,
since $\dtv(V,W)=\frac{1}{2}\norm{V-W}$ provided that 
$V(\frakX)=W(\frakX)$, which in concrete situations is often the 
case. 
With the usual operations of real scalar multiplication, addition,  
together with convolution and the total variation norm, $\calm$ is a 
real commutative Banach algebra with unity $\dirac_0$,
see for example Section 2 in \citet{MR900332}. 
For $V\in\calm$ and a power series $g(z)=\sum_{m=0}^\infty a_mz^m$ 
with $a_m\in\RR$ converging absolutely for each complex $z\in\CC$ 
with 
$\ab{z}\leq\norm{V}$, we set $g(V)=\sum_{m=0}^\infty a_mV^m\in \calm$.
The exponential of $V\in\calm$ is defined by the finite signed measure
\begin{align*}
\ee^V=\exp(V)\,=\,\sum_{m=0}^\infty\frac{V^m}{m!}\in\calm.
\end{align*}
In particular, $\CPo(t,F):=\exp(t(F-\dirac_0))$ is the compound 
Poisson distribution with parameters $t\in[0,\infty)$,
$F\in\calf$. 
In other words, this is the distribution of the random sum 
$\sum_{j=1}^NX_j$, where $N$, 
$X_j$, $(j\in\NN)$ are independent random variables,
$N$ has values in $\Zpl:=\{0,1,2,\dots\}$ and has  
Poisson distribution 
$\Po(t):=\exp(t(\dirac_1-\dirac_0))=\CPo(t,\dirac_1)$ 
with mean $t$, 
whereas the $\frakX$-valued $X_j$ are identically distributed 
with distribution $F$. We denote  the counting 
density of $\Po(t)$ 
by $\po(\bfcdot,t):\,\ZZ\longrightarrow[0,1]$, where
$\ZZ$ is the set of integers, 
$\po(m,t)=\ee^{-t}\frac{t^m}{m!}$ for $m\in\Zpl$ 
and $\po(m,t)=0$ otherwise.
If $F$ and $G$ are non-negative measures on $(\frakX,\cala)$ and $F$ 
is absolutely continuous with respect to $G$, we write $F\ll G$.
For a set $A$, let $\bbone_{A}(x)=1$ if $x\in A$ and 
$\bbone_{A}(x)=0$ otherwise. 
Set $\set{0}=\emptyset$ and $\set{n}=\{1,\dots,n\}$ for 
$n\in\NN=\{1,2,\dots\}$; further, for  $n\in\Zpl$, 
set $\setn{n}=\{0,\dots,n\}$. For a finite set $J$, 
let $\card{J}$ be the number of its elements. 
Always, let $0^0=1$, $\frac{1}{0}=\infty$ and, for $k\in\ZZ$, 
$\sum_{m=k}^{k-1}=0$ be the empty sum. 
For $k\in\NN$, let 
$\set{k}^k_{\neq}
=\{(\ell_1,\dots,\ell_k)\in\set{k}^k\,|\,\ell_i\neq \ell_j
\mbox{ for all }i,j\in\set{k}\mbox{ with }i\neq j\}$ 
be the set of all permutations on the set $\set{k}$. 
We use the standard multi-index notation: For $d\in\NN$, 
$z=(z_1,\dots,z_d)\in\CC^d$ and $m=(m_1,\dots,m_d)\in\Zpl^d$,
set $z^m=\prod_{r=1}^dz_r^{m_r}$, 
$\vecsum{m}=\sum_{r=1}^dm_r$ and $m!=\prod_{r=1}^dm_r!$.
%%%%%%%%%%%%%%%%%%%%%%%%%%%%%%%%%%%%%%%%%%%%%%%%%%%%%%%%%%%%%%%%%%%%%%
%%%%%%%%%%%%%%%%%%%%%%%%%%%%%%%%%%%%%%%%%%%%%%%%%%%%%%%%%%%%%%%%%%%%%%
\subsection{On the method used} 
%%%%%%%%%%%%%%%%%%%%%%%%%%%%%%%%%%%%%%%%%%%%%%%%%%%%%%%%%%%%%%%%%%%%%%
%%%%%%%%%%%%%%%%%%%%%%%%%%%%%%%%%%%%%%%%%%%%%%%%%%%%%%%%%%%%%%%%%%%%%%
In our proofs, we make use of the fact that $\calm$ is a 
real commutative Banach algebra with unity. In the (compound) 
Poisson approximation such an approach has already been 
used by various authors. For example, \citet{MR0142174}, 
\citet{MR0383483}, \citet{MR1360203}, \citet{MR2141351} 
considered measures on a general measurable Abelian 
group. Other authors also used Banach algebra properties under other 
assumptions, see, for instance, \citet{MR0165555} using a complex 
variable approach, \citet{MR832029} using an operator semigroup 
framework, and \citet{MR1067026} for a unification of these methods. 

Our proofs are based on ideas of \citet{MR0165555} in combination 
with 
arguments given in \cite{MR1701409} as well as new auxiliary norm 
estimates. For further papers using Kerstan's method, see 
H.\ \cite{MR0190979}, U. \cite{MR0179820}, \cite{MR3130419}, 
\cite{MR3270603}, and also some of the references cited therein. 

The main idea of \cite{MR0165555} was to expand the difference 
of two univariate distributions in a certain way and to estimate the 
norm terms involved using the Cauchy integral formula. 
In \citet{MR1701409}, the corresponding multidimensional 
generalization was studied, which made it necessary to slightly 
modify the expansion. The norm terms have been estimated using the 
Cauchy-Schwarz inequality without using integrals. 
In the present paper, we use a different expansion (see formulas (
\ref{e1649865}) and (\ref{e729875}) below) and use new norm term 
estimates using Charlier polynomials and the Cauchy-Schwarz inequality
(see Subsection \ref{s5395}). 

We note that in this paper it suffices to consider 
measures on a measurable Abelian semigroup with zero element
rather than a measurable Abelian group. 
This makes it possible to use our results in the Poisson point 
process approximation, see Section~\ref{s1545}.
%%%%%%%%%%%%%%%%%%%%%%%%%%%%%%%%%%%%%%%%%%%%%%%%%%%%%%%%%%%%%%%%%%%%%%
%%%%%%%%%%%%%%%%%%%%%%%%%%%%%%%%%%%%%%%%%%%%%%%%%%%%%%%%%%%%%%%%%%%%%%
\subsection{Review of some known results} 
%%%%%%%%%%%%%%%%%%%%%%%%%%%%%%%%%%%%%%%%%%%%%%%%%%%%%%%%%%%%%%%%%%%%%%
%%%%%%%%%%%%%%%%%%%%%%%%%%%%%%%%%%%%%%%%%%%%%%%%%%%%%%%%%%%%%%%%%%%%%%
Let us consider some important results for discrete distributions 
on $(\frakX,+,\cala)=(\RR^d,+,\calb^d)$ for $d\in\NN$, where 
$\calb^d$ is the Borel $\sigma$-Algebra over $\RR^d$.
During this subsection, let $n\in\NN$ and, 
for $j\in\set{n}$ and $r\in\set{d}$, 
\begin{gather*}  
p_j,q_{j,r}\in[0,1] \mbox{ with } \sum_{r=1}^dq_{j,r}=1 \mbox{ and }
\lambda_r=\sum_{j=1}^np_jq_{j,r}>0, \quad
\lambda=\sum_{j=1}^np_j=\sum_{r=1}^d\lambda_r,\\
U_r=\dirac_{\unitvec{r}},\quad
Q_j=\sum_{r=1}^d q_{j,r}U_r,\quad
F_j=\dirac_0+p_j(Q_j-\dirac_0),\\
Q=\frac{1}{\lambda}\sum_{j=1}^np_jQ_j
=\frac{1}{\lambda}\sum_{r=1}^d\lambda_rU_r,\quad
F=\prod_{j=1}^n F_j,\quad 
G=\CPo(\lambda,Q)=\exp(\lambda(Q-\dirac_0)).
\end{gather*}
Here,  $\unitvec{r}\in\RR^d$ is the unit vector with $1$ at 
position $r$ and $0$ otherwise. In what follows, we discuss some 
bounds in the approximation of the distribution~$F$ by~$G$.
%%%%%%%%%%%%%%%%%%%%%%%%%%%%%%%%%%%%%%%%%%%%%%%%%%%%%%%%%%%%%%%%%%%%%%
\subsubsection{The one-dimensional case $d=1$}
Here, we have $Q_1=\dots=Q_n=Q=\dirac_1$, 
such that $F=\prod_{j=1}^n (\dirac_0+p_j(\dirac_1-\dirac_0))$
is a so-called Bernoulli convolution and
$G=\Po(\lambda)$ is the Poisson distribution with 
mean $\lambda$. In this situation, one of the most remarkable 
results is the following: 
\begin{equation}\label{bh}
\frac{1}{7}\min\Bigl\{\frac{1}{\lambda},\,1\Bigr\}\sum_{j=1}^np_{j}^2
\leq\norm{F-G}
\leq 2\frac{1-\ee^{-\lambda}}{\lambda}\sum_{j=1}^np_{j}^2
\leq 2\min\Bigl\{\frac{1}{\lambda},\,1\Bigr\}\sum_{j=1}^np_{j}^2.
\end{equation}
The upper bounds of $\norm{F-G}$ are 
due to \citet[Theorem~1]{MR755837}, who used Stein's 
method to improve results of \citet[Theorem 2]{MR0142174},
\citet[formula (1) on page 174]{MR0165555}
and \citet[formula (4.23)]{MR0428387}. In their Theorem~2, they also 
showed a comparable lower bound with constant $\frac{1}{16}$ instead 
of $\frac{1}{7}$. The lower bound with the better constant was 
mentioned in Remark 3.2.2 of \citet{MR1163825}. 
The estimates in (\ref{bh}) depend on the behavior of the so-called 
magic  
factor $\frac{1}{\lambda}$ (cf.\ Introduction in \cite{MR1163825}) 
and 
on the smallness of all~$p_j$, $j\in\set{n}$, which is reflected 
by~$\sum_{j=1}^np_j^2$. It is easily seen that the leading constant 
$2$ in front of $\theta:=\frac{1}{\lambda}\sum_{j=1}^np_j^2$, resp.\ 
in front of $\lambda\theta$, in the upper bound (\ref{bh}) is 
optimal. 
However, formula (32) in \citet{MR1735783} implies that 
\begin{align}\label{e458476}
\Ab{\norm{F-G}-\sqrt{\frac{2}{\pii\ee}}\,\theta}
\leq C\,\theta\,\min\Bigl\{1,\,\frac{1}{\sqrt{\lambda}}+\theta\Bigr\},
\end{align}
where $C$ denotes an absolute constant. 
In particular, this implies that  
$\norm{F-G}
\sim\sqrt{\frac{2}{\pii\ee}}\,\theta$ as $\theta\to0$ and 
$\lambda\to\infty$.
Here,~$\sim$ means that the quotient of both sides tends to one.
We note that the bound (\ref{e458476}) is a generalization, resp.\
refinement, of results of \citet[Theorem 2]{MR0056861} and 
\citet[Theorem 1.2]{MR832029}; see also \citet[page 2]{MR1163825}. 
In \citet[formula (30)]{MR2251442}, it was shown that, in the case 
$\theta<1$,  
\begin{align}\label{e732762543}
\norm{F-G}
\leq \frac{3\theta}{2\ee(1-\sqrt{\theta})^{3/2}},
\end{align}
which is an improvement of formula (10) in \citet{MR1841404}. 
The more general Theorem~1, resp.\ Corollary~1, of the latter paper 
implies the sharpness of the constant $\frac{3}{2\ee}$. In fact, we 
have 
\begin{align}\label{e352436}
\lim_{t\downarrow0}\Bigl(\sup\frac{1}{\theta}\norm{F-G}\Bigr)
=\frac{3}{2\ee},
\end{align}
where the $\sup$ is taken over all $n\in\NN$, 
$p_1,\dots,p_n\in[0,1]$ such that $\lambda=\sum_{j=1}^np_j>0$ and 
$\theta=\frac{1}{\lambda}\sum_{j=1}^np_j^2\leq t$
(or, alternatively, such that $\max_{j\in\set{n}}p_j\leq t$). 
%%%%%%%%%%%%%%%%%%%%%%%%%%%%%%%%%%%%%%%%%%%%%%%%%%%%%%%%%%%%%%%%%%%%%%
\subsubsection{The multi-dimensional case $d\in\NN$.}
Here 
\begin{align*}
F=\prod_{j=1}^n\Bigl(\dirac_0
+p_j\sum_{r=1}^dq_{j,r}(\dirac_{\unitvec{r}}-\dirac_0)\Bigr)
\end{align*}
is a generalized multinomial distribution, which we wish to 
approximate by a product of Poisson distributions 
\begin{align*}
G=\CPo(\lambda,Q)
=\exp\Bigl(\sum_{r=1}^d\lambda_r(\dirac_{\unitvec{r}}-\dirac_0)\Bigr)
=\bigotimes_{r=1}^d\exp(\lambda_r(\dirac_{1}-\dirac_0))
=\bigotimes_{r=1}^d\Po(\lambda_r),
\end{align*}
i.e.\ $G$ is a multivariate Poisson distribution with mean vector 
$(\lambda_1,\dots,\lambda_d)$. In this context, there are two papers 
by \citet{MR0160279} and U. \citet{MR0179820}, 
which unfortunately have been largely overlooked in subsequent 
publications. Both papers considered more general convolution 
factors. Under our assumptions, some of the results are as follows.
\citet[formula~(1) on page~102]{MR0160279} used direct 
calculations to show a multivariate version of Proposition 1
of \citet[]{MR0142174}. His inequality  reads as
\begin{align}\label{e24648}
\norm{F-G}
\leq 2\sum_{j=1}^np_j^2
\end{align} 
and was later rediscovered by \citet[Theorem 1]{MR595721} using 
coupling arguments. 
We note that \citet[formulas (2) and (3) on page 102]{MR0160279} also 
proved two bounds for the point metric; 
one of these however can, under the present 
assumptions, be replaced by a bound of a better order,
cf.\ \citet[Theorem~2]{MR1679016}.
U.\ \citet[formula (0) on page 18]{MR0179820} proved a bound 
containing a magic factor by using the method of \citet{MR0165555}: 
If $\max_{j\in\set{n}}p_j\leq\frac{1}{4}$, then  
\begin{align}\label{e73457}
\norm{F-G}
\leq 9\sum_{j=1}^np_j^2
\Bigl(\sum_{r=1}^d\frac{q_{j,r}}{\sqrt{\lambda_r}}\Bigr)^2.
\end{align} 
Consequently, in view of (\ref{bh}), we see that, in order to obtain 
a new bound, which is of the right order in the case $d=1$, one could 
simply take the minimum of the right-hand sides of (\ref{e24648}) 
and (\ref{e73457}). But, as is shown below, it is possible to get 
bounds having a better structure concerning the minimum term. 
Indeed, the following interesting bound containing a magic factor was 
shown by \citet[Theorem 1]{MR974580} using Stein's method:
\begin{align}\label{e9488676}
\norm{F-G}
\leq 2\sum_{j=1}^np_j^2\min\Bigl\{c_\lambda\sum_{r=1}^d
\frac{q_{j,r}^2}{\lambda_r},\,1\Bigr\},
\end{align}
where $c_\lambda=\frac{1}{2}+\max\{\log(2\lambda),0\}$. 
Unfortunately, the term $c_\lambda$ is logarithmically 
increasing in $\lambda$ and therefore the upper bound in 
(\ref{e9488676}) does not have the correct order in the case $d=1$, 
see (\ref{bh}). 

An improvement of (\ref{e9488676}) without logarithmic factor 
was shown in \citet[Theorem 1]{MR1701409} using some modifications in 
the method of \citet{MR0165555}. Let 
\begin{gather}
g(z)=\frac{2\ee^z}{z^2}(\ee^{-z}-1+z)
=2\sum_{m=2}^\infty \frac{m-1}{m!}z^{m-2}, \quad(z\in\CC),
\label{e86576476}\\
\alpha_0=\sum_{j=1}^n g(2p_j)p_j^2\min\Bigl\{\frac{1}{2^{3/2}}
\sum_{r=1}^d\frac{q_{j,r}^2}{\lambda_r},\,1\Bigr\},\quad 
\beta_0=\sum_{j=1}^n p_j^2\min\Bigl\{\sum_{r=1}^d\frac{q_{j,r}^2}{
\lambda_r},\,1\Bigr\}.\label{e96811}
\end{gather}
We note that
\begin{align}\label{e23976}
1\leq g(x)\leq \ee^x \quad (x\in[0,\infty)),\qquad 
\max_{j\in\set{n}}g(2p_j)\leq g(2)\leq 4.195.
\end{align}
If $\alpha_0<\frac{1}{2\,\ee}$, then
\begin{equation}\label{e2345876}
\norm{F-G}
\leq \frac{2\alpha_0}{1-2\,\alpha_0\,\ee}.
\end{equation}
The following estimate is valid without any restrictions:
\begin{equation}\label{e2345877}
\norm{F-G}\leq 17.6\,\beta_0.
\end{equation}
%%%%%%%%%%%%%%%%%%%%%%%%%%%%%%%%%%%%%%%%%%%%%%%%%%%%%%%%%%%%%%%%%%%%%%
It is clear that (\ref{e2345876}) or (\ref{e2345877}) 
should be preferred over (\ref{e73457}), because of the term 
$\sum_{r=1}^d\frac{q_{j,r}^2}{\lambda_r}$ in the representations of 
$\alpha_0$ and $\beta_0$. Indeed, if 
$q_{j,r}=\frac{1}{d}$ for all $j\in\set{n}$ and $r\in\set{d}$, then 
$\lambda_1=\dots=\lambda_d$ and hence
$(\sum_{r=1}^d\frac{q_{j,r}}{\sqrt{\lambda_r}})^2
=\frac{1}{\lambda_1}
=d\sum_{r=1}^d\frac{q_{j,r}^2}{\lambda_r}$, so 
that the difference in the order is the factor $d$, if we consider
the first entry in the minimum terms in (\ref{e96811}).
On the other hand, for a precise comparison of (\ref{e2345876}) 
with (\ref{e73457}), let us assume that 
$\max_{j\in\set{n}}p_j\leq\frac{1}{4}$ and that
$\gamma:=\sum_{j=1}^np_j^2
(\sum_{r=1}^d\frac{q_{j,r}}{\sqrt{\lambda_r}})^2
<\frac{2}{9}$, such that the right-hand side 
of (\ref{e73457}) is smaller than the trivial bound $2$. 
If we now use the crude estimate
$2\alpha_0\leq g(\frac{1}{2})\frac{\gamma}{\sqrt{2}}$, 
then, since $\gamma\leq \frac{2}{9}$, (\ref{e2345876}) implies the 
bound $\frac{5}{2}\gamma$, which is better than the one in 
(\ref{e73457}).

In the present paper, among other results, we show the following 
further improvement of (\ref{e2345876}) and (\ref{e2345877}).
%%%%%%%%%%%%%%%%%%%%%%%%%%%%%%%%%%%%%%%%%%%%%%%%%%%%%%%%%%%%%%%%%%%%%%
\begin{theorem}\label{t721649}
Let the function $g$ be defined as in \textup{(\ref{e86576476})}. 
Write 
\begin{align}\label{e38274}
\alpha_1=\sum_{j=1}^ng(2p_j)p_j^2\sum_{r=1}^dq_{j,r}
\min\Bigl\{\frac{q_{j,r}}{2^{3/2}\lambda_r},\,2\Bigr\},\quad
\beta_1=\sum_{j=1}^np_j^2\sum_{r=1}^dq_{j,r}
\min\Bigl\{\frac{q_{j,r}}{\lambda_r},\,1\Bigr\}. 
\end{align}
If $\alpha_1<\frac{1}{2^{3/2}}$, then 
\begin{align}\label{e7465293}
\norm{F-G}\leq \frac{2\alpha_1}{1-2^{3/2}\alpha_1}. 
\end{align}
Without any restrictions, we have
\begin{align}\label{e7153547}
\norm{F-G}\leq  15.6\,\beta_1.
\end{align}
\end{theorem}
%%%%%%%%%%%%%%%%%%%%%%%%%%%%%%%%%%%%%%%%%%%%%%%%%%%%%%%%%%%%%%%%%%%%%%
\begin{remark} \label{r637865} 
\hspace{2em} % this was inserted in order to avoid hyperref error!
             % (link to Remark did not work otherwise) 
\begin{enumerate}

\item[(a)] 
Let us explain the bounds in Theorem \ref{t721649} with the help of 
random variables. We assume the notation as given above. 
Furthermore, for $j\in\set{n}$, let $X_j=(X_{j,1},\dots,X_{j,d})$ be 
$d$-dimensional independent Bernoulli random vectors 
with $P(X_j=(0,\dots,0))=1-p_j$ and $P(X_j=\unitvec{r})=p_jq_{j,r}$
for $r\in\set{d}$. 
Let $T=(T_1,\dots,T_d)$, where $T_r$, $(r\in\set{d})$ are independent 
one-dimensional Poisson $\Po(\lambda_r)$ distributed random variables.
Let $P^{S_n}$ and $P^{T}$ denote the distribution of 
$S_n=(S_{n,1},\dots,S_{n,d})=\sum_{j=1}^n X_j$ and $T$, respectively. 
Then $F=P^{S_n}$, $G=P^T$ and
\begin{align*}
\qquad\quad \dtv(P^{S_n},P^T)\leq \frac{\alpha_1}{1-2^{3/2}\alpha_1},
\mbox{ if } \alpha_1<\frac{1}{2^{3/2}}; \quad 
\dtv(P^{S_n},P^T)\leq 7.8\,\beta_1.
\end{align*}

\item[(b)] The structure of the term  $\beta_1$ is better than that 
of $\beta_0$, since we always have 
$\beta_1\leq \beta_0$ and there are examples in which 
$\beta_1$ is significantly smaller than $\beta_0$ 
(see Example~\ref{ex632434}). In particular, (\ref{e7153547}) is 
always better than (\ref{e2345877}). It should be mentioned that, 
if $q_{1,r}=\dots=q_{n,r}$ for all $r\in\set{d}$, 
then $\lambda_r=q_{1,r}\lambda$ for all $r\in\set{d}$ and 
$\beta_1=\min\{\frac{1}{\lambda},\,1\}\sum_{j=1}^np_j^2
=\beta_0$. 
Similarly, the structure of $\alpha_1$ is better 
than  that of $\alpha_0$. However,  $\alpha_1$ is not always smaller 
than or equal to $\alpha_0$, since $\alpha_1$ contains an additional 
factor $2$ in the second entry of the minimum term. 

\item[(c)] In practical applications, (\ref{e7465293}) often leads 
to smaller values than (\ref{e7153547}). 

\item[(d)] Generally, an inequality 
$\norm{F-G}
\leq C\,d^{c}\beta_1'$
with 
\begin{align*}
\beta_1'=\sum_{j=1}^np_j^2\sum_{r=1}^d q_{j,r}^2
\min\Bigl\{\frac{1}{\lambda_r},\,1\Bigr\}
\end{align*}
and absolute constants $C\in(0,\infty)$ and $c\in[0,1)$ cannot 
hold, see the remark after Corollary~1 in \citet{MR1679016}. 
Consequently, there is no hope of a bound of order 
$\beta_1'$. 

\item[(e)]
In view of (\ref{e86576476}), we see that, 
if $j\in\set{n}$ and $p_j$ is small, then $g(2p_j)\approx1$.
Hence, if $\alpha_1$ and $\max_{j\in\set{n}}p_j$  are small, then 
\begin{align*}
\norm{F-G}\leq c \sum_{j=1}^np_j^2\sum_{r=1}^dq_{j,r}
\min\Bigl\{\frac{q_{j,r}}{2^{3/2}\lambda_r},\,2\Bigr\}
\end{align*}
with $c\approx 2$. In (\ref{e7153547}), the factor $15.6$ cannot 
be replaced by a constant smaller than $2$, which follows from the 
remark after (\ref{bh}). Relation (\ref{e352436}) implies that 
(\ref{e7465293}) cannot generally hold when 
the factor $\frac{1}{2^{3/2}}$ in the 
representation of $\alpha_1$ (see (\ref{e38274})) is replaced
by a constant smaller than~$\frac{3}{4\ee}$. 

\item[(f)] All upper bounds in (\ref{bh}), (\ref{e732762543}), 
(\ref{e24648}), (\ref{e73457}), (\ref{e9488676}), (\ref{e2345876}), 
(\ref{e2345877}), (\ref{e7465293}) and (\ref{e7153547}) remain valid, 
if, in the definition of $F$ and $G$, we generalize $U_r$ to 
$U_r\in\calf$ for $r\in\set{d}$, which follows from the definition of 
the total variation norm, see, e.g., \citet
[page~187]{MR0199871} or \citet[page~167]{Michel1987}.

\end{enumerate}
\end{remark}
%%%%%%%%%%%%%%%%%%%%%%%%%%%%%%%%%%%%%%%%%%%%%%%%%%%%%%%%%%%%%%%%%%%%%%
\noindent
The next proposition provides lower bounds in the multi-dimensional 
case. 
%%%%%%%%%%%%%%%%%%%%%%%%%%%%%%%%%%%%%%%%%%%%%%%%%%%%%%%%%%%%%%%%%%%%%%
\begin{proposition}\label{p47296}
Let $J\subseteq\set{d}$, $y_j=\sum_{r\in J}q_{j,r}$, 
$\widetilde{p}_j=p_jy_j$ for all $j\in\set{n}$ and 
$\widetilde{\lambda}=\sum_{j=1}^n\widetilde{p}_j$. Then 
\begin{align}\label{e527}
\norm{F-G}
\geq \Norm{\prod_{j=1}^n(\dirac_0+\widetilde{p}_j(\dirac_1-\dirac_0))
-\Po(\widetilde{\lambda})}
\geq \frac{1}{7}\min\Bigl\{\frac{1}{\widetilde{\lambda}},\,1\Bigr\}
\sum_{j=1}^n\widetilde{p}_{j}^2.
\end{align}
In particular,
\begin{align}\label{e787329}
\norm{F-G}
&\geq \frac{1}{7}\min\Bigl\{\frac{1}{\lambda},\,1\Bigr\}
\sum_{j=1}^np_{j}^2
\end{align}
and 
\begin{align}\label{e76285}
\norm{F-G}\geq\frac{1}{7}\max_{r\in\set{d}}
\Bigl(\min\Bigl\{\frac{1}{\lambda_r},1\Bigr\}
\sum_{j=1}^np_j^2q_{j,r}^2\Bigr).
\end{align} 
\end{proposition}
%%%%%%%%%%%%%%%%%%%%%%%%%%%%%%%%%%%%%%%%%%%%%%%%%%%%%%%%%%%%%%%%%%%%%%
\noindent 
The second inequality in (\ref{e527}) is taken from (\ref{bh}). 
In the case $J=\set{d}$, 
the first lower bound in (\ref{e527}) is the same as the one in 
\citet[Remark 2.5]{MR935295}, who used a maximal coupling for a 
proof. 
The generalization with arbitrary $J\subseteq\set{d}$ is shown 
analogously. However, in order to keep the paper self-contained,
we give a further simple proof, which avoids the coupling method, 
see Section \ref{s275}.
The bounds in (\ref{e787329}) and (\ref{e76285}) 
follow from (\ref{e527}) with $J=\set{d}$ and 
$J=\{r\}$ for all $r\in\set{d}$, respectively. 
Consequently, the lower bound in (\ref{bh}) still holds in the 
multi-dimensional case.
The bound in (\ref{e76285}) is a slight improvement of
the first inequality in Corollary 1 of \citet{MR1679016}.

%%%%%%%%%%%%%%%%%%%%%%%%%%%%%%%%%%%%%%%%%%%%%%%%%%%%%%%%%%%%%%%%%%%%%%
Let us compare the bounds in (\ref{e7153547}),  (\ref{e787329}) and
(\ref{e76285}).
\begin{remark}
\begin{enumerate}

\item[(a)] Suppose that, for all $r\in\set{d}$, $a_r,b_r\in(0,1]$
exist, such that $a_r\leq q_{j,r}\leq b_r$
for all $j\in\set{n}$. Then 
$\min\{\frac{1}{\lambda},\,1\}\sum_{j=1}^np_j^2
\geq \frac{\beta_1}{\eta}$ 
with $\eta=\max_{r\in\set{d}}\frac{b_r}{a_r}$. 
Here, the bounds in (\ref{e7153547}) and 
(\ref{e787329}) differ at most by a constant multiple of 
$\frac{1}{\eta}$.  
If $q_{1,r}=\dots=q_{n,r}=a_r=b_r$ for all $r\in\set{d}$, 
then $\eta=1$. We note that, in this case, 
(\ref{e76285}) is worse than (\ref{e787329}). 

\item[(b)] Assume now that $c\in(0,1)$, $\kappa\in[0,\infty)$, 
$d=n$, $p_j=\frac{c}{j^\kappa}$, $q_{j,r}=\bbone_{\{j\}}(r)$ 
for all $j,r\in\set{n}$. 

Let us first assume that $\kappa=1$. Then 
(\ref{e76285}) implies that $\norm{F-G}
\geq \frac{1}{7}\max_{j\in\set{n}}p_j^2=\frac{c^2}{7}$, 
whereas (\ref{e7153547}) gives 
$\norm{F-G}\leq 15.6\sum_{j=1}^n p_j^2\leq 15.6\frac{\pii^2}{6}c^2$. 
Hence, in this case,  (\ref{e7153547}) and (\ref{e76285})
have the same order as $c\to0$. The bound (\ref{e787329}) 
gives $\norm{F-G}
\geq \frac{c^2}{7}\min\{\frac{1}{c\sum_{j=1}^n1/j},1\}
\sum_{j=1}^n\frac{1}{j^2}$, which is worse than (\ref{e76285})
as $n\to\infty$ if $c$ is fixed. 
This together with (a) shows  that the bounds in (\ref{e787329}) and 
(\ref{e76285}) are not comparable in general.

Let us now consider the case $\kappa=0$. Then 
(\ref{e787329}) and (\ref{e76285}) imply that 
$\norm{F-G}\geq \frac{1}{7}\min\{\frac{1}{nc},1\}nc^2$ and 
$\norm{F-G}\geq \frac{c^2}{7}$, respectively, 
whereas (\ref{e7153547}) gives $\norm{F-G}\leq 15.6nc^2$,
having a different order as $n\to\infty$ if $c$ is fixed.
\end{enumerate}
\end{remark}
%%%%%%%%%%%%%%%%%%%%%%%%%%%%%%%%%%%%%%%%%%%%%%%%%%%%%%%%%%%%%%%%%%%%%%
%%%%%%%%%%%%%%%%%%%%%%%%%%%%%%%%%%%%%%%%%%%%%%%%%%%%%%%%%%%%%%%%%%%%%%
\section{Main results}\label{mainres}
%%%%%%%%%%%%%%%%%%%%%%%%%%%%%%%%%%%%%%%%%%%%%%%%%%%%%%%%%%%%%%%%%%%%%%
%%%%%%%%%%%%%%%%%%%%%%%%%%%%%%%%%%%%%%%%%%%%%%%%%%%%%%%%%%%%%%%%%%%%%%
\begin{theorem}\label{t759375}
Let $d,n\in\NN$ and $\ell\in\setn{n}$. For $j\in\set{n}$ and 
$r\in\set{d}$, let 
\begin{gather*}
p_j,q_{j,r}\in[0,1] \mbox{ with } 
\sum_{r=1}^dq_{j,r}=1 \mbox{ and }
\lambda_r=\sum_{j=1}^np_jq_{j,r}>0,\quad
\lambda=\sum_{j=1}^np_j,\quad
U_r\in\calf,\\
Q_j=\sum_{r=1}^dq_{j,r}U_r,\quad
Q=\frac{1}{\lambda}\sum_{j=1}^np_jQ_j,\quad
R_j=p_j(Q_j-\dirac_0), \quad
F_j=\dirac_0+R_j, \\
V_j=F_j\ee^{-R_j}-\dirac_0,\quad F=\prod_{j=1}^nF_j.
\end{gather*}
For $k\in\setn{n}$, let
\begin{align*}
M_k&=\sum_{J\subseteq\set{n}:\,\card{J}=k}\prod_{j\in J} V_j,\qquad 
H_k=M_k\exp(\lambda(Q-\dirac_0))
\end{align*}
and set $G_\ell=\sum_{k=0}^\ell H_k$. 
Let the function $g$ be defined as in \textup{(\ref{e86576476})}. 
Write 
\begin{align*}
\alpha_1&=\sum_{j=1}^ng(2p_j)p_j^2\sum_{r=1}^dq_{j,r}
\min\Bigl\{\frac{q_{j,r}}{2^{3/2}\lambda_r},\,2\Bigr\}.
\end{align*}
If $\alpha_1<\frac{1}{2^{3/2}}$, then 
\begin{align}\label{e5254745}
\norm{F-G_\ell}
\leq \frac{\sqrt{(2(\ell+1))!}}{(\ell+1)!}2^{(\ell+1)/2}
\frac{\alpha_1^{\ell+1}}{1-2^{3/2}\alpha_1}. 
\end{align}
\end{theorem}
%%%%%%%%%%%%%%%%%%%%%%%%%%%%%%%%%%%%%%%%%%%%%%%%%%%%%%%%%%%%%%%%%%%%%%

\begin{remark} Consider the assumptions of Theorem \ref{t759375}.
In order to give an alternative formula for $G_\ell$ for the first 
few $\ell\in\setn{n}$, let $\Gamma_k=\sum_{j=1}^nV_j^k$ 
for $k\in\NN$. Then $M_0=\dirac_0$ and, for $k\in\set{n}$, 
Newton's identity (see \cite[A.IV.70, Lemma~4]{MR1080964}) gives
\begin{align*}
M_k=\frac{1}{k}\sum_{j=1}^k(-1)^{j-1}M_{k-j}\Gamma_j.
\end{align*}
In particular, if $n\geq3$, 
\begin{align*}
M_1=\Gamma_1, \quad M_2=\frac{1}{2}(\Gamma_1^2-\Gamma_2), \quad
M_3=\frac{1}{6}(\Gamma_1^3-3\Gamma_1\Gamma_2+2\Gamma_3)
\end{align*}
and consequently
\begin{gather*}
G_0=\exp(\lambda(Q-\dirac_0)), \quad
G_1=(\dirac_0+\Gamma_1)\exp(\lambda(Q-\dirac_0)),\\
G_2=\Bigl(\dirac_0+\Gamma_1
+\frac{1}{2}(\Gamma_1^2-\Gamma_2)\Bigr)\exp(\lambda(Q-\dirac_0)),\\
G_3=\Bigl(\dirac_0+\Gamma_1+\frac{1}{2}(\Gamma_1^2-\Gamma_2)
+\frac{1}{6}(\Gamma_1^3-3\Gamma_1\Gamma_2+2\Gamma_3)
\Bigr)\exp(\lambda(Q-\dirac_0)).
\end{gather*}
We note that, in \cite[formulas (10), (28)]{MR1701409}, the signed 
measure 
\begin{align}\label{e725475}
\Bigl(\dirac_0-\frac{1}{2}\sum_{j=1}^nR_j^2\Bigr)
\exp(\lambda(Q-\dirac_0))
\end{align}
as approximation of $F$ was used. The corresponding total variation 
bound has a somewhat complicated form and is of 
worse order than $\beta_0^2$, the definition of which can be found 
in (\ref{e96811}). In comparison, our signed measure 
\begin{align*}
G_1=\Bigl(\dirac_0+\sum_{j=1}^n(F_j\ee^{-R_j}-\dirac_0)\Bigr)
\exp(\lambda(Q-\dirac_0))
\end{align*}
is slightly more complicated than 
(\ref{e725475}), but gives a total variation bound of 
order~$\alpha_1^2$. 
\end{remark}
%%%%%%%%%%%%%%%%%%%%%%%%%%%%%%%%%%%%%%%%%%%%%%%%%%%%%%%%%%%%%%%%%%%%%%
\noindent 
In the following result, we present approximation bounds without a 
singularity as in (\ref{e5254745}). 
%%%%%%%%%%%%%%%%%%%%%%%%%%%%%%%%%%%%%%%%%%%%%%%%%%%%%%%%%%%%%%%%%%%%%%
\begin{theorem}\label{t729587}
Let the notation of Theorem \ref{t759375} be valid. 
Let $D_1'=3.11$ and $D_k'=D_k(\frac{g(2)}{2})^k$ 
for $k\in\NN\setminus\{1\}$, 
where $D_k$ is defined as in Corollary \ref{c6645838} below
(see Table~2). Let 
$h_1(x)=h_{1,\ell}(x)=\sum_{k=\ell+1}^\infty D_k'x^k$, 
$h_2(x)=h_{2,\ell}(x)=2+\sum_{k=1}^\ell D_k'x^k$ for 
$x\in[0,\infty)$. Write
\begin{align*}
\beta_1=\sum_{j=1}^np_j^2\sum_{r=1}^dq_{j,r}
\min\Bigl\{\frac{q_{j,r}}{\lambda_r},\,1\Bigr\}.
\end{align*}
Without any restrictions, we have
\begin{align}\label{e87965}
\norm{F-G_\ell}\leq c_\ell\,\beta_1^{\ell+1},
\end{align}
where $c_\ell=\frac{h_2(x_\ell)}{x_\ell^{\ell+1}}$ and 
$x_\ell\in(0,\infty)$ is the unique positive solution of the equation
$h_1(x_\ell)=h_2(x_\ell)$. In particular, we have 
$c_0\leq15.6$, $c_1\leq113.0$, $c_2\leq633.8$, $c_3\leq3204.8$,
$c_4\leq15945.6$.
\end{theorem}
%%%%%%%%%%%%%%%%%%%%%%%%%%%%%%%%%%%%%%%%%%%%%%%%%%%%%%%%%%%%%%%%%%%%%%

\begin{remark}
Theorem \ref{t721649} is a direct consequence of Theorems 
\ref{t759375} and \ref{t729587} for $\ell=0$.
\end{remark}

%%%%%%%%%%%%%%%%%%%%%%%%%%%%%%%%%%%%%%%%%%%%%%%%%%%%%%%%%%%%%%%%%%%%%%
\begin{example}\label{ex632434}
In order to compare the bounds in a numerical example, let us 
consider the assumptions of Theorem \ref{t759375}, \ref{t729587}
with  $d=n=1000$ and $p_{j,r}=\frac{10^{-4}}{\ab{j-r}^{1/2}+0.1}$,
$p_j=\sum_{r=1}^dp_{j,r}$ and $q_{j,r}=\frac{p_{j,r}}{p_j}$
for $j,r\in\set{n}$. This implies that 
$\beta_0= 0.081578\dots$, $\beta_1= 0.022183\dots$, 
$\alpha_0= 0.044626\dots$, $\alpha_1= 0.023037\dots$, 
$\lambda=9.01\dots$,
$\max_{j\in\set{n}}p_j= 0.009521\dots$. Here
$\alpha_0$ and $\beta_0$ are as in (\ref{e96811}).
Table 1 below shows that, in the approximation of $F$ by 
$G_0=\exp(\lambda(Q-\dirac_0))$, the bound in (\ref{e7465293})
is the smallest one.  In the approximation 
of $F$ by the signed measure $G_\ell$ for $\ell\in\set{4}$, 
we expect that the accuracy increases as $\ell$ increases. 
Indeed, this is reflected in the bounds as well. 
Furthermore, we see that here (\ref{e5254745}) is better than 
(\ref{e87965}). 
\begin{center}
\begin{tabular}{cc|ccc}
\multicolumn{5}{l}{Table 1: Numerical comparison of the bounds
in Example \ref{ex632434}}\\ 
\hline
\multicolumn{2}{c|}{Approximation by $\exp(\lambda(Q-\dirac_0))$}
& \multicolumn{3}{c}{Approximation by signed meas.\ $G_\ell$, 
$(\ell\in\set{4})$}\\ \hline
number of formula   &  upper bound & 
number of formula & $\ell$ &  upper bound \\ \hline \hline
(\ref{e24648})    & $0.163157$ & (\ref{e5254745}) & 1 & $0.002782$ \\
(\ref{e73457})    & $81.3$ & (\ref{e87965})   & 1 & $0.055608$ \\
(\ref{e9488676})  & $0.163157$ & (\ref{e5254745}) & 2 & 
$0.000166$ \\
(\ref{e2345876})  & $0.117843$ & (\ref{e87965})   & 2 & $0.006919$ \\
(\ref{e2345877})  & $1.435779$ & (\ref{e5254745}) & 3 & 
$0.000011$ \\
(\ref{e7465293})  & $0.049286$ & (\ref{e87965})   & 3 & $0.000777$ \\
(\ref{e7153547})  & $0.346060$ & (\ref{e5254745}) & 4 & 
$6.24\times10^{-7}$ \\
                  &           & (\ref{e87965})   & 4 & $0.000086$ \\ 
\hline
\end{tabular}
\end{center}
We note that the value of the bound in (\ref{e73457}) exceeds by far 
the trivial bound $2$, which however depends on the kind of example. 
The lower bounds in (\ref{e787329}) and (\ref{e76285}) 
give $0.001292$ and $1.60\times10^{-7}$, respectively. 
\end{example}

%%%%%%%%%%%%%%%%%%%%%%%%%%%%%%%%%%%%%%%%%%%%%%%%%%%%%%%%%%%%%%%%%%%%%%

\begin{remark} \label{r14143}
Let the notation of Theorems \ref{t759375}, \ref{t729587}
be valid and assume that there exist pairwise disjoint sets 
$A_1,\dots,A_d\in\cala$ with
$U_r(\frakX\setminus A_r)=0$ for all $r\in\set{d}$. 
Let $\bbone_{A_r}:\,\frakX\longrightarrow\frakX$ be the indicator 
function of $A_r$. Then, for $j\in\set{n}$, 
$f_j:=\sum_{r=1}^d\frac{\lambda}{\lambda_r}q_{j,r}\bbone_{A_r}$
is a $Q$-density of $Q_j$, since, for $B\in\cala$, we have 
\begin{align*}
\int_B f_j\,\dd Q=\sum_{r=1}^d\int_{B\cap A_r} 
\frac{\lambda}{\lambda_r}q_{j,r}
\,\dd \Bigl(\frac{\lambda_r}{\lambda}U_r\Bigr)
=\sum_{r=1}^dq_{j,r}U_r(B)=Q_j(B).
\end{align*}
Furthermore, for $c\in(0,\infty)$,  
\begin{align*}
\sum_{r=1}^dq_{j,r}
\min\Bigl\{c\frac{q_{j,r}}{\lambda_r},\,1\Bigr\}
=\sum_{r=1}^d\int_{A_r}\frac{\lambda}{\lambda_r}q_{j,r}
\min\Bigl\{c\frac{q_{j,r}}{\lambda_r},\,1\Bigr\}\,\dd Q
=\int f_j\min\Bigl\{c\frac{f_j}{\lambda},\,1\Bigr\}\,\dd Q,
\end{align*}
such that
\begin{align*}
\alpha_1&=\sum_{j=1}^ng(2p_j)p_j^2\int f_j
\min\Bigl\{\frac{f_j}{2^{3/2}\lambda},\,2\Bigr\}\,\dd Q,\quad
\beta_1=\sum_{j=1}^np_j^2\int f_j
\min\Bigl\{\frac{f_j}{\lambda},\,1\Bigr\}\,\dd Q.
\end{align*}
\end{remark}
%%%%%%%%%%%%%%%%%%%%%%%%%%%%%%%%%%%%%%%%%%%%%%%%%%%%%%%%%%%%%%%%%%%%%%
The next result shows that Theorems \ref{t759375} and \ref{t729587} 
can be generalized using the ideas of Remark \ref{r14143}. 
In fact, the $Q_j$, $(j\in\set{n})$ are now general probability 
measures and the $U_r$ for $r\in\set{d}$ are no longer needed. 
%%%%%%%%%%%%%%%%%%%%%%%%%%%%%%%%%%%%%%%%%%%%%%%%%%%%%%%%%%%%%%%%%%%%%%
\begin{corollary}\label{c7435328}
Let $n\in\NN$, $\ell\in\setn{n}$. For $j\in\set{n}$, let  
$p_j\in(0,1]$, $Q_j\in\calf$, 
$R_j=p_j(Q_j-\dirac_0)$, $F_j=\dirac_0+R_j$ and 
$V_j=F_j\ee^{-R_j}-\dirac_0$. 
Set $\lambda=\sum_{j=1}^np_j$, $F=\prod_{j=1}^nF_j$, 
$Q=\frac{1}{\lambda}\sum_{j=1}^np_jQ_j$. 
For $j\in\set{n}$, let $f_j$ be a Radon-Nikod\'{y}m 
density of $Q_j$ with respect to $Q$, which exists since 
$Q_j\ll Q$. For $k\in\setn{n}$, let 
\begin{align*}
M_k&=\sum_{J\subseteq\set{n}:\,\card{J}=k}\prod_{j\in J} V_j,\qquad 
H_k=M_k\exp(\lambda(Q-\dirac_0))
\end{align*}
and set $G_\ell=\sum_{k=0}^\ell H_k$. 
Let the function $g$ be defined as in \textup{(\ref{e86576476})}. Set
\begin{align*}
\widetilde{\alpha}_1=\sum_{j=1}^ng(2p_j)p_j^2\int f_j
\min\Bigl\{\frac{f_j}{2^{3/2}\lambda},\,2\Bigr\}\,\dd Q,\quad
\widetilde{\beta}_1=\sum_{j=1}^np_j^2\int f_j
\min\Bigl\{\frac{f_j}{\lambda},\,1\Bigr\}\,\dd Q.
\end{align*}
If $\widetilde{\alpha}_1<\frac{1}{2^{3/2}}$, then 
\begin{align}\label{e5254745b}
\norm{F-G_\ell}
\leq \frac{\sqrt{(2(\ell+1))!}}{(\ell+1)!}2^{(\ell+1)/2}
\frac{\widetilde{\alpha}_1^{\ell+1}}{1-2^{3/2}\widetilde{\alpha}_1}. 
\end{align}
The following bound is generally valid: 
\begin{align}\label{e87965b}
\norm{F-G_\ell}\leq c_\ell\,\widetilde{\beta}_1^{\ell+1},
\end{align}
where the constant $c_\ell$ is the same as in Theorem \ref{t729587}.
\end{corollary}

%%%%%%%%%%%%%%%%%%%%%%%%%%%%%%%%%%%%%%%%%%%%%%%%%%%%%%%%%%%%%%%%%%%%%%

%\pagebreak
\begin{remark} \label{r2456}
\hspace{2em} % this was inserted in order to avoid hyperref error!
             % (link to Remark did not work otherwise) 
\begin{enumerate}
\item[(a)] Let us explain the bounds of Corollary \ref{c7435328}
in the case $\ell=0$ with the help of random variables.
We assume the notation as in that corollary. Let 
$\{0\}\in\cala$ and $S_n=\sum_{j=1}^nX_j$ be the sum of independent 
$\frakX$-valued random variables $X_1,\dots,X_n$ with
$P(X_j\neq0)=p_j>0$ and $Q_j=P(X_j\in\bfcdot\,|\,X_j\neq 0)$. 
Let $T=\sum_{m=1}^NY_m$, where $N,Y_{m}$, $(m\in\NN)$ are independent 
random variables, $N$ is $\Zpl$-valued and has Poisson distribution 
$\Po(\lambda)$, whereas the $\frakX$-valued $Y_m$ are identically 
distributed with distribution $Q$. Then we have 
\begin{align*}
\qquad\quad
\dtv(P^{S_n},P^T)
\leq \frac{\widetilde{\alpha}_1}{1-2^{3/2}\widetilde{\alpha}_1},
\mbox{ if } \widetilde{\alpha}_1<\frac{1}{2^{3/2}}; \quad 
\dtv(P^{S_n},P^T)\leq 7.8\,\widetilde{\beta}_1.
\end{align*}

\item[(b)] For $\ell=0$, (\ref{e5254745b}) and (\ref{e87965b}) are 
refinements of (10) and (11) in \cite{MR2310979}.

\item[(c)]
Let the assumptions of Corollary \ref{c7435328}
hold. Further suppose  that $Q_j\ll\mu$ for all $j\in\set{n}$, 
where $\mu$ is a $\sigma$-finite measure on $(\frakX,\cala)$. 
Let $\widetilde{f}_j$ be a Radon-Nikod\'{y}m density 
of $Q_j$ with respect to $\mu$. Then 
$\widetilde{f}:=\frac{1}{\lambda}\sum_{j=1}^np_j\widetilde{f}_j$
is a $\mu$-density of $Q$. For $j\in\set{n}$, we get a 
$Q$-density of $Q_j$ by defining 
$f_j(x)=\frac{\widetilde{f}_j(x)}{\widetilde{f}(x)}$ for
$x\in\{\widetilde{f}>0\}$ and $f_j(x)=0$ otherwise. 
This gives the possibility to evaluate $\widetilde{\alpha}_1$ and 
$\widetilde{\beta}_1$ by using $\widetilde{f}_j$  for $j\in\set{n}$, 
$\widetilde{f}$ and $\mu$. In fact, 
\begin{align*}
\widetilde{\alpha}_1
&=\sum_{j=1}^ng(2p_j)p_j^2\int_{\{\widetilde{f}>0\}} 
\widetilde{f}_j\min\Bigl\{\frac{\widetilde{f}_j}{2^{3/2}
\lambda\widetilde{f}},\,2\Bigr\}\,\dd \mu,\\
\widetilde{\beta}_1
&=\sum_{j=1}^np_j^2\int_{\{\widetilde{f}>0\}} \widetilde{f}_j
\min\Bigl\{
\frac{\widetilde{f}_j}{\lambda\widetilde{f}},\,1\Bigr\}\,\dd \mu.
\end{align*}
If, for example, $(\frakX,+,\cala)=(\RR^1,+,\calb^1)$,
$\mu=\bblambda^1$ is the Lebesgue measure on $(\RR^1,\calb^1)$
and $Q_j$ is the exponential distribution with $\bblambda^1$-density 
$\widetilde{f}_j(x)=t_j\ee^{-t_jx}\bbone_{(0,\infty)}(x)$
for $x\in\RR$, $j\in\set{n}$, $t_j\in(0,\infty)$, then we obtain
\begin{align*}
\widetilde{\alpha}_1=\sum_{j=1}^ng(2p_j)p_j^2\int_{(0,\infty)} 
t_j\ee^{-t_jx}\min\Bigl\{\frac{t_j\ee^{-t_jx}}{2^{3/2}
\sum_{i=1}^np_i t_i\ee^{-t_ix}},\,2\Bigr\}\,\dd\bblambda^1(x)
\end{align*}
and a similar formula for $\widetilde{\beta}_1$. 
\end{enumerate}
\end{remark}

%%%%%%%%%%%%%%%%%%%%%%%%%%%%%%%%%%%%%%%%%%%%%%%%%%%%%%%%%%%%%%%%%%%%%%
%%%%%%%%%%%%%%%%%%%%%%%%%%%%%%%%%%%%%%%%%%%%%%%%%%%%%%%%%%%%%%%%%%%%%%
\section{Application in the Poisson point process 
approximation}\label{s1545}
%%%%%%%%%%%%%%%%%%%%%%%%%%%%%%%%%%%%%%%%%%%%%%%%%%%%%%%%%%%%%%%%%%%%%%
%%%%%%%%%%%%%%%%%%%%%%%%%%%%%%%%%%%%%%%%%%%%%%%%%%%%%%%%%%%%%%%%%%%%%%
Let $(S,\cals)$ be a measurable space and 
$\frakX=\frakX(S,\cals)$ be the set of all point measures 
of the form $\mu=\sum_{i\in I}\dirac_{x_i}$, 
where $I\subseteq\NN$ and $x_i\in S$ for all $i\in I$. 
Further, let $\cala=\sigma((\pi_B\,|\,B\in\cals))$ be the 
smallest $\sigma$-algebra over $\frakX$ such that all the 
evaluation maps 
$\pi_B:\,\frakX\longrightarrow\Zpl\cup\{\infty\}=:\overlZpl$,
$\mu\mapsto\mu(B)$ for $B\in\cals$ are measurable
with respect to the power set $2^{\overlZpl}$ of $\overlZpl$, see 
e.g.\ \citet{MR1199815}. The mapping
$\fraks:\frakX\times \frakX\longrightarrow\frakX$, 
$(\mu,\nu)\mapsto\mu+\nu$ is measurable with respect to 
$\cala\otimes\cala$ and $\cala$. Indeed, 
for $B\in\cals$ and $k\in\Zpl$, we have 
$\fraks^{-1}(\pi_B^{-1}(\{k\}))
=\bigcup_{j\in\setn{k}}
(\pi_B^{-1}(\{j\})\times\pi_B^{-1}(\{k-j\}))\in\cala\otimes\cala$ and 
$\fraks^{-1}(\pi_B^{-1}(\{\infty\}))
=(\pi_B^{-1}(\{\infty\})\times\frakX)
\cup
(\frakX\times \pi_B^{-1}(\{\infty\}))\in\cala\otimes\cala$. 
Therefore, $\fraks^{-1}(\pi_B^{-1}(C))\in\cala\otimes\cala$ for 
all $C\subseteq\overlZpl$.
Consequently, $(\frakX,+,\cala)$ is a measurable Abelian semigroup, 
where the zero element is the zero measure~$0$.

Let $n\in\NN$ be fixed and $N_j$,  $X_j$, $X_{j,k}$, $Z_j$,
$(j\in\set{n},\,k\in\NN)$ 
be independent random variables, where the $X_j$, $X_{j,k}$ are 
$S$-valued with distributions 
$P^{X_j}=P^{X_{j,k}}$, the $Z_j$ 
are Bernoulli random variables with $P(Z_j=1)=1-P(Z_j=0)=p_j\in(0,1]$
and the $N_j$ are Poisson $\Po(p_j)$ distributed. 
Suppose that, for all $j\in\set{n}$, 
$P^{X_j}$ has a density $\widetilde{h}_j$
with respect to a $\sigma$-finite measure $\nu$ on $(S,\cals)$.
Set $Q_j=P^{\dirac_{X_j}}$ for $j\in\set{n}$ and 
let $\lambda=\sum_{j=1}^np_j$,
$Q=\frac{1}{\lambda}\sum_{j=1}^np_jQ_j$,
$\eta=\frac{1}{\lambda}\sum_{j=1}^np_jP^{X_j}$ and
$\widetilde{h}=\frac{1}{\lambda}\sum_{j=1}^np_j\widetilde{h}_j$.
Then the point process $\xi=\sum_{j=1}^nZ_j\dirac_{X_j}$ has 
distribution $F=\prod_{j=1}^n(\dirac_0+p_j(Q_j-\dirac_0))$. 
The approximating $G=\exp(\lambda(Q-\dirac_0))$ is the
distribution of the Poisson point process
$\zeta=\sum_{j=1}^n\sum_{k=1}^{N_j}\dirac_{X_{j,k}}$
with intensity measure $E\zeta=E\xi=\lambda\eta$
with $\nu$-density $\lambda\widetilde{h}$.

%%%%%%%%%%%%%%%%%%%%%%%%%%%%%%%%%%%%%%%%%%%%%%%%%%%%%%%%%%%%%%%%%%%%%%
\begin{proposition}\label{c6724665}
Under the assumptions above, we have
\begin{align}\label{e21847}
\dtv(P^\xi,P^\zeta)
\leq \frac{\widetilde{\alpha}_1}{1-2^{3/2}\widetilde{\alpha}_1},
\mbox{ if } \widetilde{\alpha}_1<\frac{1}{2^{3/2}}; \quad 
\dtv(P^\xi,P^\zeta)
\leq 7.8\,\widetilde{\beta}_1,
\end{align}
where 
\begin{align*}
\widetilde{\alpha}_1
&=\sum_{j=1}^ng(2p_j)p_j^2\int_{\{\widetilde{h}>0\}}
\widetilde{h}_j\min\Bigl\{
\frac{\widetilde{h}_j}{2^{3/2}\lambda\widetilde{h}},\,2\Bigr\}
\,\dd\nu,\\
\widetilde{\beta}_1
&=\sum_{j=1}^np_j^2\int_{\{\widetilde{h}>0\}} \widetilde{h}_j
\min\Bigl\{\frac{\widetilde{h}_j}{\lambda\widetilde{h}},\,1\Bigr\}
\,\dd\nu,
\end{align*}
and $g$ is defined as in \textup{(\ref{e86576476})}. 
\end{proposition}
%%%%%%%%%%%%%%%%%%%%%%%%%%%%%%%%%%%%%%%%%%%%%%%%%%%%%%%%%%%%%%%%%%%%%%
\begin{remark}
In the literature, there are two inequalities, which 
are comparable with those of Proposition \ref{c6724665}. 
The simple one is the Le Cam type bound 
\begin{align}\label{e52645}
\dtv(P^\xi,P^\zeta)\leq\sum_{j=1}^np_j^2
\end{align}
and is comparable to (\ref{e24648}). A proof can, for example, 
be found in \citet[1.11.2 on p.\ 81]{MR0517931}. 

A more interesting bound is given in Theorem 2 of 
\citet{MR974580}, which reads in our notation as
\begin{align}\label{e2159}
\dtv(P^\xi,P^\zeta)\leq \frac{c_\lambda}{\lambda}
\sum_{j=1}^np_j^2\varphi_j^2
\end{align}
with $c_\lambda=\frac{1}{2}+\max\{\log(2\lambda),0\}$ and
\begin{align*}
\varphi_j=\sup_{C\in\cals:\,\eta(C)>0}\frac{P(X_j\in C)}{\eta(C)},
\qquad(j\in\set{n}).
\end{align*}
We note that the change of the notation is justified by
\citet[Theorem 1.4.1, p.\ 29]{MR1199815}. 
In fact, for $j\in\set{n}$, 
Barbour's term $\dirac_{Y_j}(\bfcdot\cap B)$
can be replaced with $Z_j\dirac_{X_j}$ where
$P^{Z_j}=\dirac_0+P(Y_j\in B)(\dirac_1-\dirac_0)$
and $P^{X_j}=P(Y_j\in\bfcdot\,|\,Y_j\in B)$.  

For a comparison of the bounds in (\ref{e52645}) and (\ref{e2159}) 
with those of Proposition \ref{c6724665}, we note that 
\begin{align}\label{e62759}
\widetilde{h}_j\leq\varphi_j\widetilde{h}\quad 
\nu\mbox{-almost everywhere for all } j\in\set{n}.
\end{align}
Indeed, if $C\in\cals$ with $\eta(C)>0$ then 
$\int_C\widetilde{h}_j\,\dd\nu =P(X_j\in C)
=\int_C\frac{P(X_j\in C)}{\eta(C)}\widetilde{h}\,\dd\nu
\leq \int_C\varphi_j\widetilde{h}\,\dd\nu$; 
on the other hand, if $\eta(C)=0$, then 
$\int_C\widetilde{h}_j\,\dd\nu=P(X_j\in C)=0
=\int_C\varphi_j\widetilde{h}\,\dd\nu$.
Now, (\ref{e62759}) follows from 3.17 in \citet{MR1278485}.
Therefore, 
$\int_{\{\widetilde{h}>0\}}\frac{\widetilde{h}_j^2}{\widetilde{h}}
\,\dd\nu \leq \varphi_j\leq\varphi_j^2$.
Consequently, if $\lambda$ is large and the 
$\widetilde{h}_j$ for $j\in\set{n}$ are not too different, then  
the bounds in (\ref{e21847}) are preferable to the ones in 
(\ref{e52645}) and (\ref{e2159}). 

Further results in the Poisson process approximation can, 
for example, be found in \citet[Chapter 10]{MR1163825} 
and \citet{MR1199815} and in the works cited there. 
\end{remark}
%%%%%%%%%%%%%%%%%%%%%%%%%%%%%%%%%%%%%%%%%%%%%%%%%%%%%%%%%%%%%%%%%%%%%%
%%%%%%%%%%%%%%%%%%%%%%%%%%%%%%%%%%%%%%%%%%%%%%%%%%%%%%%%%%%%%%%%%%%%%%
\section{Proofs}\label{Proofs}
\subsection{Auxiliary norm estimates}\label{s5395}
%%%%%%%%%%%%%%%%%%%%%%%%%%%%%%%%%%%%%%%%%%%%%%%%%%%%%%%%%%%%%%%%%%%%%%
%%%%%%%%%%%%%%%%%%%%%%%%%%%%%%%%%%%%%%%%%%%%%%%%%%%%%%%%%%%%%%%%%%%%%%
The proofs of the theorems require some upper bounds of certain norm 
terms, which measure the smoothness of compound Poisson 
distributions. In fact, in the simplest case terms like 
$\norm{(U-\dirac_0)^k\exp(\lambda(U-\dirac_0))}$  have to be 
considered for $U\in\calf$ and $k\in\NN$. 
For some properties of such norm terms, see, e.g., 
\cite{MR1368759}, 
\cite[Proposition~4]{MR1735783}, 
\cite[Lemma~3]{MR1841404},
\cite[Lemmata~3.4, 3.12]{MR2251442} and the references cited therein. 
In the following lemma, we present preliminary norm estimates,
which will be used in the proof of Lemma \ref{l47384}. 
A related bound can be found in \cite[Lemma~2]{MR1978097}.
%%%%%%%%%%%%%%%%%%%%%%%%%%%%%%%%%%%%%%%%%%%%%%%%%%%%%%%%%%%%%%%%%%%%%%
\begin{lemma}\label{l3519219}
Let $d,k\in\NN$, $p_{j,r}\in\RR$ for $j\in\set{k}$
and $r\in\set{d}$, 
$\Lambda=(\lambda_1,\dots,\lambda_d)\in(0,\infty)^d$. 
For $r\in\set{d}$, let $U_r\in\calf$, $W_r=U_r-\dirac_0$. Set 
$R_j=\sum_{r=1}^dp_{j,r}W_r$ for $j\in\set{k}$ and 
$G=\exp(\sum_{r=1}^d\lambda_rW_r)$. Then, we have 
\begin{align}\label{e36455}
\Norm{\Bigl(\prod_{j=1}^kR_j\Bigr)G}
&\leq\Bigl(\frac{1}{k!}\sum_{r\in\set{d}^k}\Bigl(
\sum_{\ell\in\set{k}^k_{\neq}}\prod_{j=1}^k \frac{p_{j,r_{\ell(j)}}}{
\sqrt{\lambda_{r_{\ell(j)}}}}\Bigr)^2\Bigr)^{1/2}
\leq \sqrt{k!}\prod_{j=1}^k\Bigl(\sum_{r=1}^d
\frac{p_{j,r}^2}{\lambda_{r}}\Bigr)^{1/2}.
\end{align} 
\end{lemma}
%%%%%%%%%%%%%%%%%%%%%%%%%%%%%%%%%%%%%%%%%%%%%%%%%%%%%%%%%%%%%%%%%%%%%%
\begin{proof}
We need some preparations. For $j\in\NN$, $m\in\ZZ$ and 
$t\in[0,\infty)$, let 
$\Delta^j\po(m,t)=\Delta^{j-1}\po(m-1,t)-\Delta^{j-1}\po(m,t)$,
$\Delta^0\po(m,t)=\po(m,t)$. It is well-known that 
$\Delta^j\po(m,t)=\frac{1}{t^j}\,\po(m,t)\,\Ch(j,m,t)$,
$(j,m\in\Zpl,\;t\in(0,\infty))$ (cf.\ \cite{MR1735783}), where
\begin{align*}
\Ch(j,x,t)=\sum_{i=0}^j\binomial{j}{i}\binomial{x}{i}
\,i!\,(-t)^{j-i},\qquad (j\in\Zpl,\;t,x\in\RR)
\end{align*}
denotes the Charlier polynomial of degree $j$ and
$\binomial{x}{i}=\prod_{j=1}^i\frac{x-j+1}{j}$
for $i\in\Zpl$ and $x\in \RR$.
Further, the Charlier polynomials are orthogonal with respect to the 
Poisson distribution (see, e.g., 
\cite[formula (1.14), page~4]{MR0481884}),
that is
\begin{equation}\label{eq79}
\sum_{m=0}^\infty \po(m,t)\,\Ch(i,m,t)\,\Ch(j,m,t)
=\bbone_{\{i\}}(j)\,i!\,t^i,\quad (i,j\in\Zpl,\,t\in(0,\infty)).
\end{equation}
It is easily shown that, for $j\in\Zpl$ and $r\in\set{d}$, we have
\begin{align*}
W_r^{j}\exp(\lambda_rW_r)
=\sum_{m=0}^\infty\Delta^{j}\po(m,\lambda_r)\,U_r^{m}.
\end{align*}
For $r\in\set{d}^k$ and $s\in\set{d}$, let 
$v_s(r)=\sum_{j=1}^k\bbone_{\{s\}}(r_j)$ and set
$v(r)=(v_1(r),\dots,v_d(r))\in\Zpl^d$. Clearly, $\vecsum{v(r)}=k$. 
For $r\in\set{d}^k$, we obtain
\begin{align*}
\prod_{j=1}^kW_{r_j}
=\prod_{j=1}^k\Bigl(\prod_{s=1}^dW_s^{\bbone_{\{s\}}(r_j)}\Bigr)
=\prod_{s=1}^dW_s^{v_s(r)}
\end{align*}
and similarly $\prod_{j=1}^k\lambda_{r_j}=\Lambda^{v(r)}$. 
Therefore, letting $\po(m,\Lambda)=\prod_{r=1}^d\po(m_r,\lambda_r)$
for $m\in\Zpl^d$, we get
\begin{align*}
\Norm{\Bigl(\prod_{j=1}^kR_{j}\Bigr)G}
&=\Norm{\sum_{r\in\set{d}^k}\Bigl(\prod_{j=1}^k p_{j,r_j}\Bigr)
\prod_{s=1}^d(W_s^{v_s(r)}\exp(\lambda_sW_s))}\\
&=\Norm{\sum_{r\in\set{d}^k}\Bigl(\prod_{j=1}^k p_{j,r_j}\Bigr)
\sum_{m\in\Zpl^d}^\infty\Bigl(\prod_{s=1}^d
(\Delta^{v_s(r)}\po(m_s,\lambda_s)U_s^{m_s})\Bigr)}\\
&=\Norm{\sum_{m\in\Zpl^d}\po(m,\Lambda)
\sum_{r\in\set{d}^k}\frac{1}{\Lambda^{v(r)}}
\Bigl(\prod_{j=1}^k p_{j,r_j}\Bigr)
\prod_{s=1}^d(\Ch(v_s(r),m_s,\lambda_s)U_s^{m_s})}\\
&\leq\sum_{m\in\Zpl^d}\po(m,\Lambda)
\Ab{\sum_{r\in\set{d}^k}\frac{1}{\Lambda^{v(r)}}
\Bigl(\prod_{j=1}^k p_{j,r_j}\Bigr)
\prod_{s=1}^d\Ch(v_s(r),m_s,\lambda_s)}.
\end{align*}
For $j\in\set{k}$ and $r\in\set{d}$, set
$a_{j,r}=\frac{p_{j,r}}{\sqrt{\lambda_{r}}}$. 
Hence, using the Cauchy-Schwarz inequality, we obtain
\begin{align*}
\Norm{\Bigl(\prod_{j=1}^kR_{j}\Bigr)G}^2
&\leq\sum_{m\in\Zpl^d}\po(m,\Lambda)\Bigl(\sum_{r\in\set{d}^k}
\frac{1}{\Lambda^{v(r)}}\Bigl(\prod_{j=1}^k p_{j,r_j}\Bigr)
\prod_{s=1}^d\Ch(v_s(r),m_s,\lambda_s)\Bigr)^2\\
&=\sum_{r\in\set{d}^k}\sum_{\widetilde{r}\in\set{d}^k}
\frac{1}{\Lambda^{v(r)+v(\widetilde{r})}}
\prod_{j=1}^k (p_{j,r_j}p_{j,\widetilde{r}_j})\\
&\quad{}\times\prod_{s=1}^d\Bigl(
\sum_{m_s=0}^\infty\po(m_s,\lambda_s) 
\Ch(v_s(r),m_s,\lambda_s)
\Ch(v_s(\widetilde{r}),m_s,\lambda_s)\Bigr).
\end{align*}
The application of (\ref{eq79}) now gives
\begin{align*}
\Norm{\Bigl(\prod_{j=1}^kR_{j}\Bigr)G}^2
&\leq\sum_{r\in\set{d}^k}\sum_{\widetilde{r}\in\set{d}^k:\,
v(r)=v(\widetilde{r})}
\frac{1}{\Lambda^{v(r)+v(\widetilde{r})}}
\Bigl(\prod_{j=1}^k (p_{j,r_j}p_{j,\widetilde{r}_j})\Bigr)
\prod_{s=1}^d\Bigl(v_s(r)!\,\lambda_s^{v_s(r)}\Bigr)\\
&=\sum_{m\in\Zpl^d:\,\vecsum{m}=k}m!
\sum_{r\in\set{d}^k:\,v(r)=m}  
\sum_{\widetilde{r}\in\set{d}^k:\,v(\widetilde{r})=m}
\prod_{j=1}^k(a_{j,r_j}a_{j,\widetilde{r}_j})\\
&=\sum_{m\in\Zpl^d:\,\vecsum{m}=k}m!
\Bigl(\sum_{r\in\set{d}^k:\,v(r)=m}\prod_{j=1}^ka_{j,r_j}\Bigr)^2.
\end{align*}
For $z\in\CC^d$, we have 
\begin{align*}
\sum_{m\in\Zpl^d:\,\vecsum{m}=k}
\Bigl(\sum_{s\in\set{d}^k:\,v(s)=m}1\Bigr)z^m
&=\sum_{s\in\set{d}^k}\sum_{m\in\Zpl^d:\,m=v(s)}z^{v(s)}\\
&=\sum_{s\in\set{d}^k}\prod_{j=1}^kz_{s_j}
=\Bigl(\sum_{s=1}^dz_s\Bigr)^k
=\sum_{m\in\Zpl^d:\,\vecsum{m}=k}\frac{k!}{m!}z^m,
\end{align*}
which implies that, for $m\in\Zpl^d$ with $\vecsum{m}=k$,
\begin{align}\label{e451545}
\sum_{s\in\set{d}^k:\,v(s)=m}1=\frac{k!}{m!}.
\end{align}
Consequently
\begin{align*}
\Norm{\Bigl(\prod_{j=1}^kR_{j}\Bigr)G}^2
&\leq \sum_{m\in\Zpl^d:\,\vecsum{m}=k}m!
\Bigl(\sum_{r\in\set{d}^k:\,v(r)=m}\prod_{j=1}^k a_{j,r_j}\Bigr)^2\\
&=\frac{1}{k!}\sum_{m\in\Zpl^d:\,\vecsum{m}=k}
\sum_{s\in\set{d}^k:\,v(s)=m}(v(s)!)^2 
\Bigl(\sum_{r\in\set{d}^k:\,v(r)=v(s)}\prod_{j=1}^ka_{j,r_j}\Bigr)^2\\
&=\frac{1}{k!}\sum_{s\in\set{d}^k} 
\Bigl(v(s)!\sum_{r\in\set{d}^k:\,v(r)=v(s)}\prod_{j=1}^ka_{j,r_j}
\Bigr)^2.
\end{align*}
For $m\in\Zpl^d$ and $r,s\in\set{d}^k$ with $v(r)=v(s)=m$, 
it easily follows from the definition of $v(r)$ that 
$\sum_{\ell\in\set{k}^k_{\neq}}
\bbone_{\{r\}}(s_{\ell(1)},\dots,s_{\ell(k)})=m!$.
However, a more explicit proof is as follows: 
Since the left-hand side clearly only depends on $m$, we obtain by
using (\ref{e451545}) that 
\begin{align*}
\lefteqn{\sum_{\ell\in\set{k}^k_{\neq}}
\bbone_{\{r\}}(s_{\ell(1)},\dots,s_{\ell(k)})
=\frac{m!}{k!}\sum_{\widetilde{r}\in\set{d}^k:\,v(\widetilde{r})=m}
\sum_{\ell\in\set{k}^k_{\neq}}
\bbone_{\{r\}}(s_{\ell(1)},\dots,s_{\ell(k)})}\\
&=\frac{m!}{k!}\sum_{\ell\in\set{k}^k_{\neq}}
\sum_{\widetilde{r}\in\set{d}^k:\,v(\widetilde{r})=m}
\bbone_{\{\widetilde{r}\}}(s_{\ell(1)},\dots,s_{\ell(k)})
=\frac{m!}{k!}\sum_{\ell\in\set{k}^k_{\neq}}1=m!. 
\end{align*}
Hence, for 
$s\in\set{d}^k$,
\begin{align*}
\lefteqn{
v(s)!\sum_{r\in\set{d}^k:\,v(r)=v(s)}\prod_{j=1}^k a_{j,r_j}
=\sum_{r\in\set{d}^k:\,v(r)=v(s)}
\sum_{\ell\in\set{k}^k_{\neq}}\bbone_{\{r\}}
(s_{\ell(1)},\dots,s_{\ell(k)})\prod_{j=1}^ka_{j,r_j}}\\
&=\sum_{\ell\in\set{k}^k_{\neq}}
\Bigl(\sum_{r\in\set{d}^k:\,v(r)=v(s)}\bbone_{\{r\}}(s_{\ell(1)},
\dots,s_{\ell(k)})\Bigr)\prod_{j=1}^ka_{j,s_{\ell(j)}}
=\sum_{\ell\in\set{k}^k_{\neq}}
\prod_{j=1}^ka_{j,s_{\ell(j)}}.
\end{align*}
Using the Cauchy-Schwarz inequality again, 
\begin{align*}
\lefteqn{\Norm{\Bigl(\prod_{j=1}^kR_{j}\Bigr)G}^2
\leq \frac{1}{k!}\sum_{r\in\set{d}^k}\Bigl(
\sum_{\ell\in\set{k}^k_{\neq}}\prod_{j=1}^ka_{j,r_{\ell(j)}}\Bigr)^2
=\frac{1}{k!}\sum_{\ell\in\set{k}^k_{\neq}}
\sum_{\widetilde{\ell}\in\set{k}^k_{\neq}}\sum_{r\in\set{d}^k}
\prod_{j=1}^k(a_{j,r_{\ell(j)}}a_{j,r_{\widetilde{\ell}(j)}})}\\
&\leq \frac{1}{k!}\sum_{\ell\in\set{k}^k_{\neq}}
\sum_{\widetilde{\ell}\in\set{k}^k_{\neq}}
\Bigl(\sum_{r\in\set{d}^k}\prod_{j=1}^ka_{j,r_{\ell(j)}}^2\Bigr)^{1/2}
\Bigl(\sum_{r\in\set{d}^k}
\prod_{j=1}^ka_{j,r_{\widetilde{\ell}(j)}}^2\Bigr)^{1/2}\\
&= \frac{1}{k!}\sum_{\ell\in\set{k}^k_{\neq}}
\sum_{\widetilde{\ell}\in\set{k}^k_{\neq}}
\prod_{j=1}^k\Bigl(\Bigl(\sum_{r_{\ell(j)}=1}^d
a_{j,r_{\ell(j)}}^2\Bigr)\Bigl(\sum_{r_{\widetilde{\ell}(j)}=1}^d
a_{j,r_{\widetilde{\ell}(j)}}^2\Bigr)\Bigr)^{1/2} 
=k!\prod_{j=1}^k\Bigl(\sum_{r=1}^da_{j,r}^2\Bigr),\qquad
\end{align*}
which proves (\ref{e36455}).
\end{proof}
%%%%%%%%%%%%%%%%%%%%%%%%%%%%%%%%%%%%%%%%%%%%%%%%%%%%%%%%%%%%%%%%%%%%%%
 
\begin{corollary}
Under the assumptions of Lemma \ref{l3519219},
we obtain, for $k=1$, resp.\ $k=2$, that 
\begin{align}
\norm{R_1 G}
&\leq\Bigl(\sum_{r=1}^d\frac{p_{1,r}^2}{\lambda_{r}}\Bigr)^{1/2},
\label{e1543811}\\
\norm{R_1R_2G} 
&\leq\Bigl(\frac{1}{2}\sum_{(r,s)\in\set{d}^2}\frac{(p_{1,r}p_{2,s}
+p_{1,s}p_{2,r})^2}{\lambda_r\lambda_s}\Bigr)^{1/2}
\leq \sqrt{2}\prod_{j=1}^2\Bigl(\sum_{r=1}^d\frac{p_{j,r}^2}{
\lambda_r}\Bigr)^{1/2}.\label{e1543812}
\end{align}
\end{corollary}
%%%%%%%%%%%%%%%%%%%%%%%%%%%%%%%%%%%%%%%%%%%%%%%%%%%%%%%%%%%%%%%%%%%%%%
We note that (\ref{e1543811}) was shown in  
\cite[formula~(18)]{MR1701409},
whereas (\ref{e1543812}) is a generalization of one part of 
(19) of that paper.
The next lemma is needed in the proof of Lemma \ref{l47384} below.
%%%%%%%%%%%%%%%%%%%%%%%%%%%%%%%%%%%%%%%%%%%%%%%%%%%%%%%%%%%%%%%%%%%%%%
\begin{lemma}\label{l46278945}
Let $k\in\NN$, $m\in\Zpl^2$ with $\vecsum{m}\leq k$. Then
\begin{align*}
(2m_1+m_2)!\leq((2k)!)^{m_1/k}((2k-1)!)^{m_2/(2k)},
\end{align*}
where equality holds in the case $k=1$.
\end{lemma}
%%%%%%%%%%%%%%%%%%%%%%%%%%%%%%%%%%%%%%%%%%%%%%%%%%%%%%%%%%%%%%%%%%%%%%
\begin{proof} For $\ell\in\NN$, we have 
$\frac{(\ell!)^{\ell+1}}{((\ell+1)!)^{\ell}}
=\frac{\ell!}{(\ell+1)^{\ell}}\leq 1$ and 
$\frac{((2\ell-1)!)^{\ell+1}}{((2\ell+1)!)^{\ell}}
=\frac{(2\ell-1)!}{(2\ell(2\ell+1))^{\ell}}\leq 1$. 
Therefore $(\ell!)^{1/\ell}$ and $((2\ell-1)!)^{1/\ell}$
are both increasing in $\ell\in\NN$.
Hence we may assume that $m_2\geq1$ and $\vecsum{m}=k$. 
Using that $\ell!\leq \ell^{\ell-1}$ for $\ell\in\NN$, we get
\begin{align*}
\frac{((2m_1+m_2)!)^{2k}}{((2k)!)^{2m_1}((2k-1)!)^{m_2}}
&=((2k-m_2)!)^{m_2}\Bigl(\frac{(2k-m_2)!}{(2k)!}\Bigr)^{2k-2m_2}
\Bigl(\frac{(2k-m_2)!}{(2k-1)!}\Bigr)^{m_2}\\
&\leq  \frac{(2k-m_2)^{m_2(2k-m_2-1)} }{(2k-m_2)^{m_2(2k-2m_2)}
(2k-m_2)^{m_2(m_2-1)}}=1,
\end{align*}
which implies the assertion. 
\end{proof}
%%%%%%%%%%%%%%%%%%%%%%%%%%%%%%%%%%%%%%%%%%%%%%%%%%%%%%%%%%%%%%%%%%%%%%

\begin{lemma}\label{l47384}
Let the assumptions of Lemma \ref{l3519219} hold and 
let $p_j=\sum_{r=1}^d\ab{p_{j,r}}$ for $j\in\set{k}$.
Set $c_k=((2k)!)^{1/(2k)}$ and $c_k'=((2k-1)!)^{1/(4k)}$. 
Then, for all $u\in[0,\frac{1}{2}]^k$, $v,w\in(0,\infty)^k$, 
we have
\begin{align*}
\Norm{\Bigl(\prod_{j=1}^kR_j^2\Bigr)G}
&\leq \prod_{j=1}^k\Bigl(C_j\sum_{r=1}^d\ab{p_{j,r}}\,
\min\Bigl\{\frac{\ab{p_{j,r}}}{\lambda_r},\,\frac{4}{w_j} p_j\Bigr\}
\Bigr),
\end{align*} 
where $C_j=\max\{
c_k+c_k'\frac{u_j}{v_j},\,(2(1-u_j)+c_k'u_jv_j)w_j\}$ for 
$j\in\set{k}$. In particular, for $u=0$ and $w_j=\frac{c_k}{2}$ for 
$j\in\set{k}$, we obtain
\begin{align*}
\Norm{\Bigl(\prod_{j=1}^kR_j^2\Bigr)G}
\leq\sqrt{(2k)!}\prod_{j=1}^k\Bigl(
\sum_{r=1}^d\ab{p_{j,r}}\,\min\Bigl\{\frac{\ab{p_{j,r}}}{\lambda_r},\,
\frac{8}{c_k}p_j\Bigr\}\Bigr).
\end{align*} 
\end{lemma}
We note that $c_k\geq8$, if $k\geq 10$. 
%%%%%%%%%%%%%%%%%%%%%%%%%%%%%%%%%%%%%%%%%%%%%%%%%%%%%%%%%%%%%%%%%%%%%%
\begin{proof} 
We may assume that $p_j>0$ for all $j\in\set{k}$; further, set
\begin{gather*}
I_j=\Bigl\{r\in\set{d}\,\Big|\,\frac{\ab{p_{j,r}}}{\lambda_r}
\leq\frac{4}{w_j}p_j\Bigr\},\qquad
I_j^\compl=\set{d}\setminus I_j,\\
a_j=\sum_{r\in I_j}\frac{p_{j,r}^2}{\lambda_r}
=\sum_{r\in I_j}\ab{p_{j,r}}\,
\min\Bigl\{\frac{\ab{p_{j,r}}}{\lambda_r},\,\frac{4}{w_j}p_j\Bigr\},\\
b_j=2\sum_{r\in I_j^\compl}\ab{p_{j,r}}
=\frac{w_j}{2p_j}\sum_{r\in I_j^\compl}\ab{p_{j,r}}
\min\Bigl\{\frac{\ab{p_{j,r}}}{\lambda_r},\,\frac{4}{w_j}p_j\Bigr\},\\
R_j'=\sum_{r\in I_j}p_{j,r}W_r
=\sum_{r=1}^d\bbone_{I_j}(r)p_{j,r}W_r, \quad 
R_j''=\sum_{r\in I_j^\compl}p_{j,r}W_r,\\
Y_j=2(1-u_j)R_j'R_j''+(R_j'')^2.
\end{gather*}
In particular, we have 
\begin{gather*}
b_j\leq 2p_j,\quad 
a_j+\frac{2p_j}{w_j}b_j
=\sum_{r=1}^d\ab{p_{j,r}}\,
\min\Bigl\{\frac{\ab{p_{j,r}}}{\lambda_r},\,\frac{4}{w_j}p_j\Bigr\},\\
\norm{R_j'}\leq 2p_j-b_j, \;
\norm{R_j''}\leq b_j, \;
\norm{Y_j}\leq 2(1-u_j)(2p_j-b_j)b_j+b_j^2 \leq 4(1-u_j)p_jb_j.
\end{gather*}
Further, for $J_1,J_2\subseteq\set{k}$ with $J_1\cap J_2=\emptyset$, 
$\card{J_1}=m_1$, $\card{J_2}=m_2$, we have  $m_1+m_2\leq k$ and 
Lemmata \ref{l3519219} and \ref{l46278945} imply that
\begin{align*}
\Norm{\Bigl(\prod_{j\in J_1}(R_j')^2\Bigr)
\Bigl(\prod_{j\in J_2}R_j'\Bigr) G}
&\leq \sqrt{(2m_1+m_2)!}\Bigl(\prod_{j\in J_1}a_j\Bigr)
\prod_{j\in J_2}\sqrt{a_j}\\
&\leq\Bigl(\prod_{j\in J_1}(c_ka_j)\Bigr)
\prod_{j\in J_2}(c_k'\sqrt{a_j}).
\end{align*}
Therefore
\begin{align*}
\lefteqn{\Norm{\Bigl(\prod_{j=1}^kR_j^2\Bigr)G}
=\Norm{\Bigl(\prod_{j=1}^k((R_j')^2+2u_jR_j'R_j''+Y_j)\Bigr)G}}\\
&\hspace{1cm}=\Norm{\sum_{m\in\Zpl^2:\,\vecsum{m}\leq k}
\sum_{J_1\subseteq\set{k}:\,\card{J_1}=m_1}
\sum_{J_2\subseteq\set{k}\setminus J_1:\,\card{J_2}=m_2}
\Bigl(\prod_{j\in J_1}(R_j')^2\Bigr)\\
&\hspace{1cm}\quad{}\times\Bigl(\prod_{j\in J_2}(2u_jR_j'R_j'')\Bigr)
\Bigl(\prod_{j\in \set{k}\setminus(J_1\cup J_2)}Y_j\Bigr)G}\\
&\hspace{1cm}\leq\sum_{m\in\Zpl^2:\,\vecsum{m}\leq k}
\sum_{J_1\subseteq\set{k}:\,\card{J_1}=m_1}
\sum_{J_2\subseteq\set{k}\setminus J_1:\,\card{J_2}=m_2}
\Norm{\Bigl(\prod_{j\in J_1}(R_j')^2\Bigr)
\Bigl(\prod_{j\in J_2}R_j'\Bigr)G}\\
&\hspace{1cm}\quad{}\times
\Bigl(\prod_{j\in J_2}(2u_j\norm{R_j''})\Bigr)
\prod_{j\in \set{k}\setminus(J_1\cup J_2)}\norm{Y_j}\\
&\hspace{1cm}\leq\sum_{m\in\Zpl^2:\,\vecsum{m}\leq k}
\sum_{J_1\subseteq\set{k}:\,\card{J_1}=m_1}
\sum_{J_2\subseteq\set{k}\setminus J_1:\,\card{J_2}=m_2}
\Bigl(\prod_{j\in J_1}(c_ka_j)\Bigr)\\
&\hspace{1cm}\quad{}\times
\Bigl(\prod_{j\in J_2}(2c_k'u_j\sqrt{a_j}\,b_j)\Bigr) 
\prod_{j\in \set{k}\setminus(J_1\cup J_2)}(4(1-u_j)p_jb_j),
\end{align*}
giving
\begin{align*}
\Norm{\Bigl(\prod_{j=1}^kR_j^2\Bigr)G}
\leq \prod_{j=1}^d\Bigl(c_k a_j
+2c_k'u_j\sqrt{\frac{a_j}{v_j}\,b_j^2v_j}+4(1-u_j)p_jb_j\Bigr).
\end{align*}
Using that $2\sqrt{xy}\leq x+y$ for $x,y\in[0,\infty)$, 
we obtain, for $j\in\set{d}$, 
\begin{align*}
\lefteqn{c_k a_j+2c_k'u_j
\sqrt{\frac{a_j}{v_j}b_j^2v_j}+4(1-u_j)p_jb_j}\\
&\hspace{1cm}\leq\Bigl(c_k+c_k'\frac{u_j}{v_j}\Bigr)a_j
+(2(1-u_j)+c_k'u_jv_j)2p_jb_j\\
&\hspace{1cm}\leq \max\Bigl\{c_k+c_k'\frac{u_j}{v_j},\,
(2(1-u_j)+c_k'u_jv_j)w_j
\Bigr\}\Bigl(a_j+\frac{2p_j}{w_j}b_j\Bigr)\\
&\hspace{1cm}=C_j\sum_{r=1}^d\ab{p_{j,r}}\,
\min\Bigl\{\frac{\ab{p_{j,r}}}{\lambda_r},\,\frac{4}{w_j}p_j\Bigr\},
\end{align*}
which implies the assertion.
\end{proof}
%%%%%%%%%%%%%%%%%%%%%%%%%%%%%%%%%%%%%%%%%%%%%%%%%%%%%%%%%%%%%%%%%%%%%%

\begin{corollary}\label{c6645838}
Under the assumptions of Lemma \ref{l47384}, we have 
\begin{align*}
\Norm{\Bigl(\prod_{j=1}^kR_j^2\Bigr)G}
&\leq  D_k \,k!\prod_{j=1}^k\Bigl(\sum_{r=1}^d\ab{p_{j,r}}\,
\min\Bigl\{\frac{\ab{p_{j,r}}}{\lambda_r},\,p_j\Bigr\}\Bigr),
\end{align*} 
where $w_j=4$, $(j\in\set{k})$, if $k\in\set{9}$,  
and the values of $u_j$, $v_j$, $D_k$ 
are given in Table 2 below.
\begin{center}
\begin{tabular}{c||c|c|c|c|c}
\multicolumn{6}{l}{\textup{Table 2: Explicit values of the 
constants $D_k$ in Corollary \ref{c6645838}}}\\ \hline
$k$ & $1$ & $2$ & $3$ & $4$ & $5$ \\ \hline \hline 
$u_j$  & $0.5000$ & $0.5000$ & $0.5000$ & $0.5000$ & $0.4500$ 
\\ \hline
$v_j$  & $0.1708$ & $0.2574$ & $0.3589$ & $0.4666$ & $0.5192$ 
\\ \hline 
$D_k$ &  $4.342$ & $10.784$ & $21.721$  & $40.687$ & $74.672$ 
\medskip\\ \hline 
$k$ & $6$ & $7$ & $8$ & $9$ & $\geq 10$\\ \hline \hline 
$u_j$ & $0.3000$ & $0.1996$ & $0.1500$ & $0.0500$ & $0$\\ \hline
$v_j$ & $0.4414$ & $0.4099$ & $0.5002$ & $0.4560$ & $1$\\ \hline 
$D_k$ &  $125.448$ & $186.872$ & $253.020$ & $305.314$ &
$\frac{\sqrt{(2k)!}}{k!}$
\end{tabular}
\end{center}
\end{corollary}
%%%%%%%%%%%%%%%%%%%%%%%%%%%%%%%%%%%%%%%%%%%%%%%%%%%%%%%%%%%%%%%%%%%%%%

\begin{lemma}\label{l261947}
Let $d\in\NN$. For $r\in\set{d}$, let $p_r\in[0,1]$, 
$\lambda_r\in(0,\infty)$ with 
$\lambda_r\geq p_r$, $U_r\in\calf$. 
We assume that $p:=\sum_{r=1}^dp_r\leq 1$. 
Let $R=\sum_{r=1}^d p_r(U_r-\dirac_0)$, 
$G=\exp(\sum_{r=1}^d \lambda_r(U_r-\dirac_0))$. 
Let $u\in[0,\frac{1}{2}]$, $v,w\in(0,\infty)$ and $w_0\in(1,\infty)$ 
be the unique solution of $f(w_0)=\frac{2}{w}$, where
$f(x)=x\log(1+\frac{1}{x-1})-1=\int_0^1\frac{t}{x-t}\,\dd t$ for 
$x\in(1,\infty)$. Then, letting 
$C=\max\{(\sqrt{2}+\frac{u}{v})\frac{2}{w},\,4(1-u)+2uv\}$, 
\begin{align*}
\norm{((\dirac_0+R)\ee^{-R}-\dirac_0)G}
\leq C \sum_{r=1}^dp_r\min\Bigl\{w_0\frac{p_r}{\lambda_r},\,p\Bigr\}.
\end{align*}
In particular, if $u=\frac{1}{2}$, $v=0.47248$ and $w=2$, then 
$C\leq 2.473$ and $w_0\leq 1.256$, giving 
\begin{align*}
\norm{((\dirac_0+R)\ee^{-R}-\dirac_0)G}
\leq 3.11\sum_{r=1}^dp_r\min\Bigl\{\frac{p_r}{\lambda_r},\,p\Bigr\}.
\end{align*}
\end{lemma}
%%%%%%%%%%%%%%%%%%%%%%%%%%%%%%%%%%%%%%%%%%%%%%%%%%%%%%%%%%%%%%%%%%%%%%
\begin{proof} We may assume that $p_r>0$ for all $r\in\set{d}$. 
It is easily shown that 
\begin{align*}
((\dirac_0+R)\ee^{-R}-\dirac_0)G=-\int_0^1 tR^2\exp(-tR)G\,\dd t,
\end{align*}
where the equality holds setwise. 
From Lemma \ref{l47384}, we obtain for $t\in(0,1)$ that
\begin{align*}
\norm{R^2\exp(-tR)G}
\leq C\frac{w}{2}
\sum_{r=1}^dp_r\min\Bigl\{\frac{p_r}{\lambda_r-tp_r},\,
\frac{4}{w}p\Bigr\}.
\end{align*}
Consequently
\begin{align*}
\norm{((\dirac_0+R)\ee^{-R}-\dirac_0)G}
&\leq C\frac{w}{2}\int_0^1 t 
\sum_{r=1}^dp_r\min\Bigl\{\frac{p_r}{\lambda_r-tp_r},\,
\frac{4}{w}p\Bigr\}\,\dd t\\
&\leq C\sum_{r=1}^dp_r
\min\Bigl\{\frac{w}{2}f\Bigl(\frac{\lambda_r}{p_r}\Bigr),\,p\Bigr\}.
\end{align*}
Let $r\in\set{d}$. If 
$\frac{w}{2}f(\frac{\lambda_r}{p_r})\leq p$, then 
$f(\frac{\lambda_r}{p_r})\leq \frac{2}{w}=f(w_0)$, 
giving 
$\frac{\lambda_r}{p_r}\geq w_0$, since $f$ is decreasing. Further,
$xf(x)=\int_0^1\frac{t}{1-t/x}\,\dd t$ is decreasing in 
$x\in(1,\infty)$, giving 
$\frac{w}{2}f(\frac{\lambda_r}{p_r})
\leq \frac{w}{2}\frac{p_r}{\lambda_r}w_0f(w_0)
=w_0\frac{p_r}{\lambda_r}$.
On the other hand, if 
$p\leq \frac{w}{2}f(\frac{\lambda_r}{p_r})$, then 
$f(w_0)
=\frac{2}{w}
\leq \frac{1}{p}f(\frac{\lambda_r}{p_r})
\leq f(p\frac{\lambda_r}{p_r})$
and so $p\frac{\lambda_r}{p_r}\leq w_0$, which implies that 
$p\leq w_0\frac{p_r}{\lambda_r}$. 
Therefore, in any case 
$\min\{\frac{w}{2}f(\frac{\lambda_r}{p_r}),\,p\}
\leq \min\{w_0\frac{p_r}{\lambda_r},\,p\}$. 
Together with the above, we obtain the assertion. 
\end{proof}
%%%%%%%%%%%%%%%%%%%%%%%%%%%%%%%%%%%%%%%%%%%%%%%%%%%%%%%%%%%%%%%%%%%%%%
%%%%%%%%%%%%%%%%%%%%%%%%%%%%%%%%%%%%%%%%%%%%%%%%%%%%%%%%%%%%%%%%%%%%%%
\subsection{Remaining proofs}\label{s275}
%%%%%%%%%%%%%%%%%%%%%%%%%%%%%%%%%%%%%%%%%%%%%%%%%%%%%%%%%%%%%%%%%%%%%%
%%%%%%%%%%%%%%%%%%%%%%%%%%%%%%%%%%%%%%%%%%%%%%%%%%%%%%%%%%%%%%%%%%%%%%
\noindent
\begin{proof}[Proof of the first inequality in 
Proposition \ref{p47296}]
Let the notation of Remark \ref{r637865}(a) be valid. 
Then $P(\sum_{r\in J}X_{j,r}=1)
=1-P(\sum_{r\in J}X_{j,r}=0)=\widetilde{p}_j$ for $j\in\set{n}$
and $\sum_{r\in J}\lambda_r=\widetilde{\lambda}$. Consequently
$\sum_{r\in J}S_{n,r}$ and $\sum_{r\in J}T_r$ have the distributions 
$\prod_{j=1}^n(\dirac_0+\widetilde{p}_j(\dirac_1-\dirac_0))$ and
$\Po(\widetilde{\lambda})$, respectively, and hence
\begin{align*}
\norm{F-G}
&=2\sup_{A\subseteq\Zpl^d}\ab{P(S_n\in A)-P(T\in A)}\\
&\geq 2\sup_{B\subseteq\Zpl}
\Ab{P\Bigl(\sum_{r\in J}S_{n,r}\in B\Bigr)
-P\Bigl(\sum_{r\in J}T_r\in B\Bigr)}\\
&=\norm{P^{\sum_{r\in J}S_{n,r}}-P^{\sum_{r\in J}T_r}}
=\Norm{\prod_{j=1}^n(\dirac_0+\widetilde{p}_j(\dirac_1-\dirac_0))
-\Po(\widetilde{\lambda})},
\end{align*}
which implies the first inequality in Proposition \ref{p47296}.
\end{proof}

%%%%%%%%%%%%%%%%%%%%%%%%%%%%%%%%%%%%%%%%%%%%%%%%%%%%%%%%%%%%%%%%%%%%%%
\begin{proof}[Proof of Theorem \ref{t759375}] We first note that 
\begin{align}\label{e1649865}
F=\prod_{j=1}^nF_j=\prod_{j=1}^n((V_j+\dirac_0)\ee^{R_j})
=\sum_{k=0}^nH_k=G_n,
\end{align}
which implies that $F-G_\ell=\sum_{k=\ell+1}^nH_k$. 
For $j\in\set{n}$, we have 
$V_j=F_j\ee^{-R_j}-\dirac_0
=-\frac{g(-R_j)}{2}R_j^2$. Hence, for $k\in\setn{n}$,
\begin{align}\label{e729875}
H_k 
&=(-1)^k\sum_{J\subseteq\set{n}:\,\card{J}=k}
\Bigl(\prod_{j\in J}\frac{g(-R_j)}{2}\Bigr)
\Bigl(\prod_{j\in J} R_j^2\Bigr)\exp(\lambda(Q-\dirac_0)).
\end{align}
If $k\in\set{n}$, $J\subseteq\set{n}$ with $\card{J}=k$, then 
Lemma \ref{l47384} with $u=0$, $w_j=\frac{1}{\sqrt{2}}$
for $j\in J$ implies that 
\begin{align}\label{e917465}
\Norm{\Bigl(\prod_{j\in J}R_j^2\Bigr)\exp(\lambda(Q-\dirac_0))}
\leq \sqrt{(2k)!}\prod_{j\in J}\Bigl(\sum_{r=1}^dp_jq_{j,r}
\min\Bigl\{\frac{p_jq_{j,r}}{\lambda_r},\,2^{5/2}p_j\Bigr\}\Bigr),
\end{align}
since $((2k)!)^{1/(2k)}\geq\sqrt{2}$. On the other hand, for 
$j\in\set{n}$, $\norm{R_j}\leq 2p_j$ and therefore
\begin{align}\label{e141432}
\norm{g(-R_j)}
=\Norm{2\sum_{m=2}^\infty \frac{m-1}{m!}(-R_j)^{m-2}}
\leq 2\sum_{m=2}^\infty \frac{m-1}{m!}\norm{R_j}^{m-2}
\leq g(2p_j).
\end{align}
By (\ref{e729875}), (\ref{e917465}), (\ref{e141432}) and the 
polynomial theorem, we derive for $k\in\set{n}$, 
\begin{align*}
\norm{H_k}
&\leq\sum_{J\subseteq\set{n}:\,\card{J}=k}
\Bigl(\prod_{j\in J} \frac{\norm{g(-R_j)}}{2}\Bigr)
\Norm{\Bigl(\prod_{j\in J}R_j^2\Bigr)\exp(\lambda(Q-\dirac_0))}\\
&\leq\frac{\sqrt{(2k)!}}{2^k}
\sum_{J\subseteq\set{n}:\,\card{J}=k}
\prod_{j\in J}\Bigl(g(2p_j)p_j^2\sum_{r=1}^dq_{j,r}
\min\Bigl\{\frac{q_{j,r}}{\lambda_r},\,2^{5/2}\Bigr\}\Bigr)\\
&\leq\frac{\sqrt{(2k)!}}{k!\,2^k}
\Bigl(\sum_{j=1}^ng(2p_j)p_j^2\sum_{r=1}^dq_{j,r}
\min\Bigl\{\frac{q_{j,r}}{\lambda_r},\,2^{5/2}\Bigr\}\Bigr)^k
=\frac{\sqrt{(2k)!}}{k!\,2^k}(2^{3/2}\,\alpha_1)^k.
\end{align*}
It is easily shown that 
$\frac{\sqrt{(2k)!}}{k!\,2^k}$ is decreasing in $k\in\Zpl$. 
Consequently, if $\alpha_1<\frac{1}{2^{3/2}}$, then
\begin{align*}
\norm{F-G_\ell}
&\leq\sum_{k=\ell+1}^n\norm{H_k}
\leq\sum_{k=\ell+1}^n\frac{\sqrt{(2k)!}}{k!\,2^k}(2^{3/2}\,
\alpha_1)^k\\
&\leq \frac{\sqrt{(2(\ell+1))!}}{(\ell+1)!\,2^{\ell+1}}2^{3(\ell+1)/2}
\frac{\alpha_1^{\ell+1}}{1-2^{3/2}\alpha_1},
\end{align*}
which proves (\ref{e5254745}). 
\end{proof}
%%%%%%%%%%%%%%%%%%%%%%%%%%%%%%%%%%%%%%%%%%%%%%%%%%%%%%%%%%%%%%%%%%%%%%
\noindent
\begin{proof}[Proof of Theorem \ref{t729587}]
We need a further bound for $\norm{H_k}$, $(k\in\set{n})$
in terms of~$\beta_1$. Lemma~\ref{l261947} gives 
\begin{align*}
\norm{H_1}
&=\Norm{\sum_{j=1}^n((\dirac_0+R_j)\ee^{-R_j}-\dirac_0)
\exp(\lambda(Q-\dirac_0))\Bigr)}
\leq  D_1'\beta_1.
\end{align*}
For $k\in\set{n}\setminus\{1\}$, Corollary \ref{c6645838} and the 
polynomial theorem imply that 
\begin{align*}
\norm{H_k}
&\leq\sum_{J\subseteq\set{n}:\,\card{J}=k}
\Bigl(\prod_{j\in J} \frac{g(2p_j)}{2}\Bigr)
\Norm{\Bigl(\prod_{j\in J}R_j^2\Bigr)\exp(\lambda(Q-\dirac_0))}\\
&\leq D_k\Bigl(\frac{g(2)}{2}\Bigr)^k\,k!
\sum_{J\subseteq\set{n}:\,\card{J}=k}
\prod_{j\in J}\Bigl(p_j^2\sum_{r=1}^dq_{j,r}
\min\Bigl\{\frac{q_{j,r}}{\lambda_r},\,1\Bigr\}\Bigr)
\leq D_k'\beta_1^k.
\end{align*}
Hence 
\begin{align*}
\norm{F-G_\ell}&\leq\sum_{k=\ell+1}^\infty D_k'\beta_1^k
=h_1(\beta_1)
\end{align*}
and, alternatively,
\begin{align*}
\norm{F-G_\ell}&\leq\norm{F}+\norm{G_\ell}
\leq 2+\sum_{k=1}^\ell\norm{H_k}
\leq 2+\sum_{k=1}^\ell D_k'\beta_1^k
=h_2(\beta_1).
\end{align*}
By the definition of $D_k'$ for $k\geq 10$, we know that 
$h_1(x)<\infty$ for $x\in[0,\frac{1}{g(2)})$.
Further, it is easily seen that $\frac{h_1(x)}{h_2(x)}$ is 
increasing in $x\in[0,\frac{1}{g(2)})$ with 
$\lim_{x\uparrow 1/g(2)}\frac{h_1(x)}{h_2(x)}=\infty$. 
Therefore, for all $\ell\in\setn{n}$, 
there exists a unique $x_\ell\in(0,\infty)$ with 
$h_1(x_\ell)=h_2(x_\ell)$. If $\beta_1\leq x_\ell$ then 
$\norm{F-G_\ell}
\leq\frac{h_1(\beta_1)}{\beta_1^{\ell+1}}\beta_1^{\ell+1}
\leq  \frac{h_1(x_\ell)}{x_\ell^{\ell+1}}\beta_1^{\ell+1}
=c_\ell\beta_1^{\ell+1}$. 
If $\beta_1>x_\ell$, then $\norm{F-G_\ell}
\leq\frac{h_2(\beta_1)}{\beta_1^{\ell+1}}\beta_1^{\ell+1}
\leq c_\ell\beta_1^{\ell+1}$.
Hence, generally we have
$\norm{F-G_\ell}\leq c_\ell\beta_1^{\ell+1}$. 
In particular,  
$x_0\in(0.128316,\,0.128317)$, 
$x_1\in(0.147522,\,0.147523)$,  
$x_2\in(0.189075,\,0.189076)$, 
$x_3\in(0.215065,\,0.215066)$, 
$x_4\in(0.226773,\,0.226774)$,
which implies the remaining part of the assertion. 
\end{proof}
%%%%%%%%%%%%%%%%%%%%%%%%%%%%%%%%%%%%%%%%%%%%%%%%%%%%%%%%%%%%%%%%%%%%%%

\begin{proof}[Proof of Corollary \ref{c7435328}] 
The proof follows arguments very similar to those used in the proofs 
of Theorems 1 and 2 in \cite{MR2310979}, where a comparable result 
was 
shown, generalizing (\ref{e2345876}) and (\ref{e2345877}). 
The idea here is a standard approximation procedure: 
In the first step, construct a new set of distributions 
$\widetilde{Q}_{1},\dots,\widetilde{Q}_{n}$ of the form 
used in Theorems \ref{t759375} and \ref{t729587}, such that all 
the norms $\norm{Q_j-\widetilde{Q}_j}$, $(j\in\set{n})$ are small. 
This also leads to corresponding new (signed) measures 
$\widetilde{F}$ and $\widetilde{G}_\ell$. In the second step,
use the properties of the total variation distance to show that 
$\norm{F-\widetilde{F}}$ and $\norm{\widetilde{G}_\ell-G_\ell}$ are 
both small. Finally, use Theorems \ref{t759375} and \ref{t729587}
to estimate $\norm{\widetilde{F}-\widetilde{G}_\ell}$ and prove that 
the resulting bounds are close to the bounds in (\ref{e5254745b}) 
and (\ref{e87965b}). We omit the details. 
\end{proof}
%%%%%%%%%%%%%%%%%%%%%%%%%%%%%%%%%%%%%%%%%%%%%%%%%%%%%%%%%%%%%%%%%%%%%%

\begin{proof}[Proof of Proposition \ref{c6724665}] 
Under the assumptions of Section \ref{s1545}, 
let $\tau:\,S\longrightarrow\frakX$, $x\mapsto \dirac_x$. 
For arbitrary $B\in\cals$, we then have 
$\pi_B\circ \tau=\bbone_{B}$
and $B=\tau^{-1}(\pi_B^{-1}(\{1\}))$, and hence
$\{\tau^{-1}(A)\,|\,A\in\cala\}=\cals$. 
In particular, $\tau$ is $\cals$-$\cala$-measurable. 
Let $\mu=\nu^\tau$ be the image measure of 
$\nu$ under $\tau$ defined on $(\frakX,\cala)$. 
For $B\in\cals$, we have 
$\{\dirac_x\,|\,x\in B\}=\pi_B^{-1}(\{1\})
\cap\pi_{S\setminus B}^{-1}(\{0\})\in \cala$ and
$\mu(\{\dirac_x\,|\,x\in B\})=\nu(B)$. This shows that,
since $\nu$ is $\sigma$-finite, this holds for $\mu$ as well. 
If $A\in\cala$ with $\mu(A)=0$, then 
$\nu(\tau^{-1}(A))=0$, and in turn 
$0=P^{X_j}(\tau^{-1}(A))=P((\tau\circ X_j)^{-1}(A))
=Q_j(A)$ and hence $Q_j\ll\mu$ for all $j\in\set{n}$. 
Let $\widetilde{f}_j$ be a Radon-Nikod\'{y}m density 
of $Q_j$ with respect to $\mu$ and set 
$\widetilde{f}=\frac{1}{\lambda}\sum_{j=1}^np_j\widetilde{f}_j$.  
As has been observed in Remark \ref{r2456}(c), 
for $j\in\set{n}$, $f_j=\frac{\widetilde{f}_j}{\widetilde{f}}
\bbone_{\{\widetilde{f}>0\}}$ is 
a Radon-Nikod\'{y}m density of $Q_j$ with respect to $Q$. 
From the above, we get that,
for each $B\in\cals$, a set $A\in\cala$ exists such that 
$B=\{\tau\in A\}$ and hence 
\begin{align*}
\int_{B}\widetilde{f}_j\circ\tau\,\dd\nu
=\int_{A}\widetilde{f}_j\,\dd\mu
=Q_j(A)=P^{X_j}(B)=\int_B \widetilde{h}_j\,\dd\nu. 
\end{align*}
Therefore $\widetilde{f}_j\circ\tau=\widetilde{h}_j$,
$\widetilde{f}\circ\tau=\widetilde{h}$ 
and $f_j\circ\tau
=\frac{\widetilde{f}_j\circ\tau}{\widetilde{f}\circ\tau}
\bbone_{\{\widetilde{f}\circ\tau>0\}}
=\frac{\widetilde{h}_j}{\widetilde{h}}
\bbone_{\{\widetilde{h}>0\}}$ $\nu$-almost everywhere. 
The assertion now follows from 
Corollary \ref{c7435328} and Remark \ref{r2456}(c) using that 
\begin{align*}
\widetilde{\beta}_1
&=\sum_{j=1}^np_j^2\int_{\{\widetilde{f}>0\}} \widetilde{f}_j
\min\Bigl\{
\frac{\widetilde{f}_j}{\lambda\widetilde{f}},\,1\Bigr\}\,\dd \mu
=\sum_{j=1}^np_j^2\int_{\{\widetilde{h}>0\}} \widetilde{h}_j
\min\Bigl\{\frac{\widetilde{h}_j}{\lambda\widetilde{h}},\,1
\Bigr\}\,\dd \nu
\end{align*}
and a similar calculation for $\widetilde{\alpha}_1$. 
\end{proof}

%%%%%%%%%%%%%%%%%%%%%%%%%%%%%%%%%%%%%%%%%%%%%%%%%%%%%%%%%%%%%%%%%%%%%%
\section*{Acknowledgment}
%%%%%%%%%%%%%%%%%%%%%%%%%%%%%%%%%%%%%%%%%%%%%%%%%%%%%%%%%%%%%%%%%%%%%%
The author would like to thank two anonymous reviewers, Andrew 
Barbour, and Lutz~Mattner for their comments which led to an improved 
version of the paper.
%%%%%%%%%%%%%%%%%%%%%%%%%%%%%%%%%%%%%%%%%%%%%%%%%%%%%%%%%%%%%%%%%%%%%%
\raggedbottom % inserted to avoid Underfull \vbox (badness 10000)
\bibliographystyle{alea3}
\bibliography{rtv049}
%%%%%%%%%%%%%%%%%%%%%%%%%%%%%%%%%%%%%%%%%%%%%%%%%%%%%%%%%%%%%%%%%%%%%%
\end{document}